\documentclass[a4paper, 10pt]{article}
\usepackage{amsmath}
\usepackage{amssymb,esint}
\usepackage{amscd}
\usepackage{xspace}
\usepackage{fancyhdr}
\providecommand{\U}[1]{\protect\rule{.1in}{.1in}}
\newtheorem{theorem}{Theorem}[section]

\newtheorem{definition}[theorem]{Definition}

\newtheorem{proposition}[theorem]{Proposition}
\newtheorem{remark}[theorem]{Remark}

\newenvironment{proof}[1][Proof]{\textbf{#1.} }{\hfill\rule{0.5em}{0.5em}}
{\catcode`\@=11\global\let\AddToReset=\@addtoreset
\AddToReset{equation}{section}

\AddToReset{theorem}{section}

\begin{document}
\title{Wiener criteria for existence of large solutions of nonlinear parabolic equations with absorption  in a non-cylindrical domain }
%Some p-Laplace Type Equations Involving Measure
\author{
%{\bf Marie-Fran\c{c}oise Bidaut-V\'eron\thanks{E-mail address: veronmf@univ-tours.fr}}\\[1mm]
 {\bf Quoc-Hung  Nguyen \thanks{ E-mail address: Hung.Nguyen-Quoc@lmpt.univ-tours.fr, quoc-hung.nguyen@epfl.ch}}\\[0.5mm]
 {\bf Laurent V\'eron\thanks{ E-mail address: Laurent.Veron@lmpt.univ-tours.fr}}\\[2mm]
{\small Laboratoire de Math\'ematiques et Physique Th\'eorique, }\\
{\small  Universit\'e Fran\c{c}ois Rabelais,  Tours,  FRANCE}}
\date{}
\maketitle

\begin{abstract}
 We obtain a necessary condition and a sufficient  condition, both expressed in terms of Wiener type tests involving the parabolic  $W_{q'}^{2,1}$- capacity, where $q'=\frac{q}{q-1}$ and $q>1$,  for the existence of large solutions to equation  $\partial_tu-\Delta u+u^q=0$ in  a non-cylindrical domain. We provide also a sufficient condition for the existence of such solutions to  equation $\partial_t u-\Delta u+e^u-1=0$.  Besides, we apply our results to equation: $\partial_tu-\Delta u+a|\nabla u|^p+bu^{q}=0$ for $a,b>0$, $1<p<2$ and $q>1$. 
\end{abstract}\smallskip
\smallskip


%\tableofcontents
\section{Introduction}
The aim of this paper is to study the problem of existence of large solutions to some nonlinear parabolic equations with superlinear absorption in an {\it arbitrary} bounded open set $O\subset\mathbb{R}^{N+1}$, $N\geq 2$. These are functions $u\in C^{2,1}(O)$, solutions of
\begin{equation}\label{6h2205201413}
\begin{array}{lll}
\partial_tu-\Delta u +|u|^{q-1}u=0 &\text{ in }~ O,\\ \phantom{-,-}
 \displaystyle  \lim\limits_{\delta\to 0}\inf_{O\cap Q_\delta(x,t)}u =\infty~~&\text{ for all}~~(x,t)\in \partial_pO,\\ 
   \end{array}
 \end{equation}
 with $q>1$ and 
\begin{equation}\label{6h2205201414}
\begin{array}{lll}
\partial_tu-\Delta u +\operatorname{sign}(u)(e^{|u|}-1)=0 &\text{ in }~ O,\\ \phantom{----,,--}
  \displaystyle     \lim\limits_{\delta\to 0}\inf_{O\cap Q_\delta(x,t)}u =\infty~~&\text{ for all}~~(x,t)\in \partial_pO,\\ 
    \end{array}
 \end{equation}
in which expressions $\partial_pO$ denotes the parabolic boundary of $O$, i.e. the set all points $X=(x,t)\in\partial O$ such that the intersection of the cylinder $Q_\delta(x,t):=B_\delta(x)\times(t-\delta^2,t)$ with $O^c$ is not empty for any $\delta>0$. By the maximal principle for parabolic equations we can assume that all solutions of \eqref{6h2205201413} and \eqref{6h2205201414} are positive. Henceforth we consider only positive solutions of the preceding equations. \\
 In \cite{66Ve1}, we studied the existence and the uniqueness of solution of semilinear heat equations in a cylindrical domain,
\begin{equation}\label{6h2205201415}
 \begin{array}{lll}
   \partial_tu-\Delta u +f(u)=0 &\text{ in }~ \Omega\times(0,\infty),\\ \phantom{\partial_tu-\Delta  +f(u)}
    u=\infty~~&\text{ in }~ \partial_p\left(\Omega\times(0,\infty)\right),\\ 
    \end{array} 
 \end{equation}
 where $\Omega$ is a bounded open set in $\mathbb{R}^N$ and $f$ a continuous nondecreasing  real-valued function such that $f(0)\geq 0$ and $f(a)>0$ for some $a>0$. In order to obtain the existence of a maximal solution of $\partial_tu-\Delta u +f(u)=0 \text{ in }~ \Omega\times(0,\infty)$ there is need to introduce the following assumptions
\begin{equation}
\label{6hKO}\begin{array} {lll}
  \displaystyle &(i)   \qquad\qquad&\displaystyle\int_{a}^{\infty}\left(\int_{0}^{s}f(\tau)d\tau\right)^{-\frac{1}{2}}ds <\infty,\\[4mm]
  \displaystyle  &(ii)  \displaystyle &\displaystyle\int_{a}^{\infty}\left(f(s)\right)^{-1}ds <\infty.
  \end{array}\end{equation}
  Condition (i), due to Keller and Osserman, is a necessary and sufficient for the existence of a maximal solution to 
 \begin{align}\label{6h230520141}
 -\Delta u+f(u)=0~~ \text{in}~\Omega.
 \end{align} 
 Condition (ii) is a necessary and sufficient for the existence of a maximal solution of the differential equation
  \begin{align}\label{6hODE}
 \varphi'+f(\varphi)=0\qquad\text{in }(0,\infty),
  \end{align} 
and this solution tends to $\infty$ at $0$.   In \cite{66Ve1}, it is shown that if 
 for any $m\in\mathbb{R}$ there exists $L=L(m)>0$ such that 
 \begin{align*}
 \text{ for any } x,y\geq m \Rightarrow f(x+y)\geq f(x)+f(y)-L, 
 \end{align*}
and if \eqref{6h230520141} has a large  solution, then \eqref{6h2205201415} admits a solution.
   
  It is not alway true that the maximal solution to \eqref{6h230520141} is a large solution. However, if $f$ satisfies 
  \begin{align*}
  \int_{1}^{\infty}s^{-2(N-1)/(N-2)}f(s)ds<\infty ~~\text{ if } N\geq 3,
  \end{align*} 
  or 
  \begin{align*}
    \inf\left\{a\geq 0:\int_{0}^{\infty} f(s)e^{-as}ds<\infty~~\right\}<\infty ~\text{ if } N=2,
    \end{align*} 
    then \eqref{6h230520141} has a large solution for any bounded domain $\Omega$, see \cite{66MV1}.\smallskip
    
 When $f(u)=u^q$, $q>1$ and $N\geq 3$, the first above condition is satisfied if and only if $q<q_c:=\frac{N}{N-2}$, this is called {\it the sub-critical case}. When $q\geq q_c$, a necessary and sufficient condition for the existence of a large solution to
 \begin{align}\label{u^q}
 -\Delta u+u^q=0~~ \text{in}~\Omega
 \end{align} 
 is  expressed in term of a Wiener-type test,
 \begin{align}\label{labut}
 \int_{0}^{1}\frac{\text{Cap}_{2,q'}(\Omega^c\cap B_r(x))}{r^{N-2}}\frac{dr}{r}=\infty~~\text{ for all }~x\in\partial\Omega. 
 \end{align}
 
 In the case $q=2$ it is obtained  by Dhersin and Le Gall \cite{66DhGa}, see also \cite{66Legall1,66Legall2}, using probabilistic methods involving the Brownian snake; this method can be extended for $1<q\leq 2$ by using ideas from \cite{66DyKu1,66DyKu2}. In the general case the result is proved by Labutin, by  purely analytic methods \cite{66Lab}. 
Note that $q'=\frac{q}{q-1}$ and $\text{Cap}_{2,q'}$ is the capacity associated to the Sobolev space $W^{2,q'}(\mathbb{R}^N)$.  

In \cite{66HV2} we obtain sufficient conditions  for the existence of a large solution to 
 \begin{align}\label{ež-1}
 -\Delta u+e^u-1=0~~ \text{in}~\Omega,
 \end{align} 
expressed in terms of the Hausdorff $\mathcal{H}_1^{N-2}$-capacity  in $\mathbb{R}^N$, and more precisely 
 \begin{align}
 \int_{0}^{1}\frac{\mathcal{H}_1^{N-2}(\Omega^c\cap B_r(x))}{r^{N-2}}\frac{dr}{r}=\infty~~\text{ for all }~x\in\partial\Omega. 
 \end{align}
 We refer to \cite{66MV2} for investigation of the initial trace theory of \eqref{6h2205201415}. \smallskip
 
In \cite{66EvGa}, Evans and Gariepy establish a Wiener  criterion for  the regularity of a boundary point (in the sense of potential theory) for the heat operator $L=\partial_t-\Delta $ in an arbitrary bounded set of $\mathbb{R}^{N+1}$. We denote by $\mathfrak M(\mathbb R^{N+1})$ the set of Radon  measures in $\mathbb{R}^{N+1}$ and, for any compact set $K\subset\mathbb{R}^{N+1}$, by $\mathfrak M_K(\mathbb{R}^{N+1})$ the subset of $\mathfrak M(\mathbb{R}^{N+1})$ of measures with support in $K$. Their positive cones are respectively denoted by 
$\mathfrak M^+(\mathbb{R}^{N+1})$ and 
$\mathfrak M^+_K(\mathbb{R}^{N+1})$.
The capacity used in this criterion is the thermal capacity defined by \begin{align*}
\text{Cap}_{\mathbb{H}}(K)=\sup\{\mu(K):\mu\in\mathfrak M^+_K(\mathbb{R}^{N+1}), \mathbb{H}*\mu\leq 1 \},
\end{align*}
for any  $K\subset\mathbb{R}^{N+1}$ compact, where $\mathbb{H}$ is the heat kernel in $\mathbb{R}^{N+1}$. 
It coincides with the parabolic Bessel $\mathcal{G}_1$-capacity $\text{Cap}_{\mathcal{G}_1,2}$,
\begin{align*}
\text{Cap}_{\mathcal{G}_1,2}(K)=\sup\left\{\int_{\mathbb{R}^{N+1}}|f|^2dxdt: f\in L^2_+(\mathbb{R}^{N+1}),~\mathcal{G}_1*f\geq \chi_{K} \right\},
\end{align*} here $\mathcal{G}_1$ is the parabolic Bessel kernel of first order, see \cite[Remark 4.12]{66H1}. Garofalo and Lanconelli \cite{66GaLa} extend this result to the parabolic operator $L=\partial_t-\text{div}(A(x,t)\nabla) $, where $A(x,t)=(a_{i,j}(x,t))$, $i,j=1,2,...,N$ is a real, symmetric, matrix-valued function on $\mathbb{R}^{N+1}$ with $C^\infty$ entries satisfying 
\begin{align*}
C^{-1}|\xi|^2\leq \sum_{i,j=1}^{N}a_{i,j}(x,t)\xi_i\xi_j\leq C|\xi|^2~~\forall (x,t)\in\mathbb{R}^{N+1}, \,\forall \xi\in\mathbb{R}^{N},
\end{align*} 
for some constant $C>0$. \smallskip

Much less is known concerning the equation 
\begin{align}\label{6h200620141}
\partial_tu-\Delta u+f(u)=0
\end{align}
in a bounded open set $O\subset$ of $\mathbb{R}^{N+1}$, where $f$ is a continuous function in $\mathbb{R}$.  Gariepy and Ziemer \cite{66GaZi, 66Zi1} prove that if there exist $(x_0,t_0)\in \partial_p O$, $l\in\mathbb{R}$ and a weak solution $u\in W^{1,2}(O)\cap L^\infty(O)$ of  \eqref{6h200620141} such that $\eta(-l-\varepsilon+u)^+, \eta(l-\varepsilon-u)^+\in W_0^{1,2}(O)$ for any $\varepsilon>0$ and $\eta\in C_c^\infty(B_r(x_0)\times(-r^2+t_0,r^2+t_0))$ for some $r>0$, and if there holds
\begin{align*}
\int_{0}^{1}\frac{\text{Cap}_{\mathbb{H}}\left(O^c\cap\left( B_\rho(x_0)\times(t_0-\frac{9}{4}\alpha\rho^2,t_0-\frac{5}{4}\alpha\rho^2)\right)\right)}{\rho^N}\frac{d\rho}{\rho}=\infty ~\text{for some }\alpha>0,
\end{align*} then $\lim\limits_{(x,t)\to (x_0,t_0)}u(x,t)=l$. 
This result is not easy to use because it is not clear whether \eqref{6h200620141} has a weak solution $u\in W^{1,2}(O)$. In this article we show that \eqref{6h200620141} admits a maximal solution $u\in C^{2,1}(O)$ in an arbitrary bounded open set $O$,  which is constructed by using an approximation of $O$ from inside by dyadic parabolic cubes, provided that $f$ is as in  \eqref{6h2205201415} and satisfies \eqref{6hKO}. 

The main purpose of this article is to extend Labutin's result \cite{66Lab} to the semilinear parabolic equation \eqref{6h2205201413}. Namely, we give a necessary and a sufficient condition for the existence of solutions to problem \eqref{6h2205201413}  in a bounded non-cylindrical domain $O\subset\mathbb{R}^{N+1}$, expressed in terms of a Wiener test based upon the parabolic $W^{2,1}_{q'}$-capacity in $\mathbb{R}^{N+1}$. We also give a sufficient condition for solving problem \eqref{6h2205201414} expressed in terms of a Wiener test based upon the parabolic Hausdorff $\mathcal{PH}_\rho^N$-capacity.  These capacities are defined as follows: if $K\subset\mathbb{R}^{N+1}$ is a compact set, we set
\begin{align*}
           \text{Cap}_{2,1,q'}(K)=\inf\{||\varphi||^{q'}_{W^{2,1}_{q'}(\mathbb{R}^{N+1})}:\varphi\in S(\mathbb{R}^{N+1}), \varphi\geq 1 \text{ in a neighborhood of}~K \}, 
           \end{align*} 
           where      
          \begin{align*}
          ||\varphi||_{W^{2,1}_{q'}(\mathbb{R}^{N+1})}=|| \varphi||_{L^{q'}(\mathbb{R}^{N+1})}+||\frac{\partial \varphi}{\partial t}||_{L^{q'}(\mathbb{R}^{N+1})}+||\nabla \varphi||_{L^{q'}(\mathbb{R}^{N+1})}+\sum\limits_{i,j} ||\frac{\partial^{2} \varphi}{\partial x_i\partial x_j}||_{L^{q'}(\mathbb{R}^{N+1})},
          \end{align*}  
           and for a Suslin set $E\subset\mathbb{R}^{N+1}$,
            \begin{align*}
                       \text{Cap}_{2,1,q'}(E)=\sup\{\text{Cap}_{2,1,q'}(D):D\subset E, D \text{ compact} \}. 
                       \end{align*}
         This capacity has been used  in order to obtain estimates expressed with the help of potential  that are most helpful for studying quasilinear parabolic equations (see e.g. \cite{66BaPi1,66BaPi2,66H1}). Thanks to a result due to Richard and Bagby \cite{66Bag}, the capacities $\text{Cap}_{2,1,p}$ and $\text{Cap}_{\mathcal{G}_2,p}$ are equivalent in the sense that, for any Suslin set $K\subset\mathbb{R}^{N+1}$, there holds 
                    \begin{align*}
                    C^{-1}\text{Cap}_{2,1,q'}(K)\leq \text{Cap}_{\mathcal{G}_2,q'}(K)\leq C\text{Cap}_{2,1,q'}(K),
                    \end{align*} 
                    for some $C=C(N,q)$, where $\text{Cap}_{\mathcal{G}_2,q'}$ is the parabolic  Bessel $\mathcal{G}_2$-capacity, see \cite{66H1}.\\
               For a set $E\subset\mathbb{R}^{N+1}$, we define $\mathcal{PH}_\rho^N(E)$ by 
              \begin{align*}
              \mathcal{PH}_\rho^N(E)=\inf\left\{\displaystyle\sum_jr_j^N: E\subset\bigcup B_{r_j}(x_j)\times(t_j-r_j^2,t_j+r_j^2),~r_j\leq\rho\right\}.
              \end{align*} 
              It is easy to see that, for $0<\sigma\leq \rho$ and  $E\subset\mathbb{R}^{N+1}$, there holds
              \begin{align}\label{6h230520142}
              \mathcal{PH}_\rho^N(E)\leq \mathcal{PH}_\sigma^N(E)\leq C(N)\left(\frac{\rho}{\sigma}\right)^2\mathcal{PH}_\rho^N(E).
              \end{align}
With these notations, we can state the two main results of this paper. 
\begin{theorem}\label{6h230520143} Let $N\geq 2$ and $q\geq q_*:=\frac{N+2}{N}$. Then \smallskip

\noindent (i)  The equation \begin{align}\label{6h220520149}
\partial_tu-\Delta u +u^q=0 \text{ in }~O
\end{align} 
admits a large solution if there holds
 \begin{align}\label{6h2205201412}\int_{0}^{1}\frac{\operatorname{Cap}_{2,1,q'}(O^c\cap (B_{\frac{\rho}{30}}(x)\times (t-30 \rho^2,t-\rho^2)))}{\rho^N}\frac{d\rho}{\rho}=\infty,~~
 \end{align}  
 for any $(x,t)\in\partial_pO$ and $q>q_*$ or  $q=q_*$ when $N\geq 3$. \smallskip

\noindent (ii)   If equation \eqref{6h220520149} admits a large solution, then 
  \begin{align}\label{6h0106201413}
  \int_{0}^{1}\frac{\operatorname{Cap}_{2,1,q'}(O^c\cap Q_\rho (x,t))}{\rho^N}\frac{d\rho}{\rho}=\infty,
  \end{align}
 for any $(x,t)\in\partial_pO$,  where $Q_\rho(x,t)=B_\rho(x)\times (t-\rho^2,t)$.

\end{theorem}

 It is an open problem to prove that the maximal solution is unique whenever it exists as it holds in the elliptic case for equation (\ref{u^q}), see Remark p. 25.

\begin{theorem}\label{6h230520144} Let $N\geq 2$. The equation \begin{align}\label{6h2205201410}
 \partial_tu-\Delta u +e^u-1=0 \text{ in }~O
 \end{align} admits a large solution if  there holds
 \begin{align}\label{6h2205201411}\int_{0}^{1}\frac{\mathcal{PH}_1^N(O^c\cap (B_{\frac{\rho}{30}}(x)\times (t-30 \rho^2,t-\rho^2)))}{\rho^N}\frac{d\rho}{\rho}=\infty,
\end{align}
for any $(x,t)\in\partial_pO$.  
\end{theorem}
From properties of the $W^{2,1}_{q'}$-capacity, relation \eqref{6h2205201412}  is satisfied if the following relations hold in which $|\;\,|$ denotes the Lebesgue measure in $\mathbb{R}^{N+1}$, 
\begin{align*}\int_{0}^{1}\frac{|O^c\cap (B_{\frac{\rho}{30}}(x)\times (t-30 \rho^2,t-\rho^2)|^{1-\frac{2q'}{N+2}}}{\rho^N}\frac{d\rho}{\rho}=\infty~\text{ when } q>q_*,
\end{align*}
and 
\begin{align*}\int_{0}^{1}\frac{\left(\log_+ \left |O^c\cap (B_{\frac{\rho}{30}}(x)\times (t-30 \rho^2,t-\rho^2)\right|^{-1}\right)^{-\frac{N}{2}}}{\rho^N}\frac{d\rho}{\rho}=\infty
~\text{ when } q=q_*.
\end{align*}
Similarly, it follows from  properties of the $\mathcal{PH}_1^N$-capacity that
identity \eqref{6h2205201411} is verified if 
\begin{align*}
\int_{0}^{1}\frac{|O^c\cap (B_{\frac{\rho}{30}}(x)\times (t-30 \rho^2,t-\rho^2)|^{\frac{N}{N+2}}}{\rho^N}\frac{d\rho}{\rho}=\infty.
\end{align*}
When $O=\{(x,t)\in \mathbb{R}^{N+1}: |x|^2+\frac{|t|^2}{\lambda}<1\}$ for some $\lambda>0$, we see that $\partial O=\partial_p O$.  Therefore  \eqref{6h0106201413} holds  for any $(x,t)\in \partial_p O$, and \eqref{6h2205201412}-\eqref{6h2205201411} hold for any $(x,t)\in \partial_p O\backslash\{(0,\sqrt{\lambda})\}$. However, \eqref{6h2205201412} and \eqref{6h2205201411} are also valid at $(x,t)=(0,\sqrt{\lambda})$ if $\lambda >1800^2$, but  not valid if $\lambda <1800^2$.\smallskip

As a consequence of Theorem \ref{6h230520143} we derive a sufficient condition for the existence of a large solution of a class of viscous parabolic Hamilton-Jacobi equations.
\begin{theorem}\label{6h260520148} Let  $q_1>1$.  If there exists a large solution $v\in C^{2,1}(O)$ of 
\begin{align*}
\partial_tv-\Delta v +v^{q_1}=0 ~~~\text{ in }~O,
\end{align*}
then,  for any $a,b>0$, $1<q<q_1$ and $1<p<\frac{2q_1}{q_1+1}$, the problem 
\begin{equation}\label{6hHJP}
\begin{array}{lll}
     \partial_tu-\Delta u +a|\nabla u|^p+bu^{q}=0\quad &\text{ in } O,\\ 
     \phantom{     \partial_tu-\Delta  +a|\nabla u|^p+bu^{q}}
      u=\infty~&\text{ on }~\partial_p O,\\      
      \end{array} 
\end{equation}
 admits a solution $u\in C^{2,1}(O)$ which satisfies 
\begin{align*}
u(x,t)\geq C\min\left\{a^{-\frac{1}{p-1}}R^{-\frac{2-p}{p-1}+\frac{2}{\alpha(q_1-1)}}, b^{-\frac{1}{q-1}}R^{-\frac{2}{q-1}+\frac{2}{\alpha(q_1-1)}}\right\}(v(x,t))^{\frac{1}{\alpha}},
\end{align*}
for all $(x,t)\in O$ where $R>0$ is such that  $O\subset \tilde{Q}_R(x_0,t_0)$,   $C=C(N,p,q,q_1)>0$ and $\alpha= \max\left\{\frac{2(p-1)}{(q_1-1)(2-p)},\frac{q-1}{q_1-1}\right\}\in (0,1)$.
\end{theorem}
%%%%%%%%%%%%%%%%%%%%%%%%%%%%%%%%%%%%%%%%%%%%%%%%%%%%%%%%%%%%%%%%%%%%%%%%%%%%%%%%%%%%%%%%%%%%%%%%%%%%%%%%%%%%%%%%%%%%%%%%%%%%%%%%%%%%%%%%%%%%%%%%%%%%%%%%%%%%%%%%%%%%%%%%%%%%%%%%%%%%%%%%%%%%%%PRELIMINARIES%%%%%%%%%%%%%%%%%%%%%%%%%%%%%%%%%%%%%%%%%%%%%%%%%%%%%%%%%%%%%%%%%%%%%%%%%%%%%%%%%%%%%%%%%%%
\section{Preliminaries}
Throughout the paper, we denote  
$$Q_\rho(x,t)=B_\rho(x)\times (t-\rho^2,t],$$ and  $$\tilde{Q}_\rho(x,t)=B_\rho(x)\times (t-\rho^2,t+\rho^2),$$ for $(x,t)\in\mathbb{R}^{N+1}$ and $\rho>0$, and $r_k=4^{-k}$ for all $k\in\mathbb{Z}$. We also denote $A\lesssim (\gtrsim )B$ if $A\leq  (\geq)C B $ for some $C$ depending on some structural constants,  $ A \asymp B$ if $A\lesssim B\lesssim A$. 

\begin{definition}\label{6hWolff} Let $R\in (0,\infty]$ and $\mu\in \mathfrak{M}^+(\mathbb{R}^{N+1})$. We define $R-$truncated Riesz parabolic potential $ \mathbb{I}_{2}^R$ of $\mu$ by
\begin{align*}
 \mathbb{I}_{2}^R[\mu](x,t)=\int_{0}^{R}\frac{\mu(\tilde{Q}_\rho(x,t))}{\rho^{N}}\frac{d\rho}{\rho}~~\text{ for all } ~(x,t)\in \mathbb{R}^{N+1},
\end{align*}                
and the $R-$truncated fractional maximal parabolic  potential $\mathbb{M}_{2}^R$ of $\mu$ by
\begin{align*}\mathbb{M}_{2}^R[\mu](x,t)=\sup_{0<\rho<R}\frac{\mu(\tilde{Q}_\rho(x,t))}{\rho^{N}}~~\text{ for all } ~(x,t)\in \mathbb{R}^{N+1}.
\end{align*}
\end{definition}
We also set $ \mathbb{I}_{2}^\infty= \mathbb{I}_{2}$ and $\mathbb{M}_{2}^\infty=\mathbb{M}_{2}$. We recall two results in \cite{66H1}.
\begin{theorem}\label{6h2305201411} Let $q>1, R>0$ and $K$ be  a compact set in $\mathbb{R}^{N+1}$. There exists $\mu:=\mu_K\in \mathfrak{M}^+(\mathbb{R}^{N+1})$ with compact support in $K$ such that
\begin{align*}
\mu (K)\asymp \operatorname{Cap}_{2,1,q'}(K)\asymp \int_{\mathbb{R}^{N+1}}\left(\mathbb{I}_2^{2R}[\mu]\right)^qdxdt
\end{align*}
 where the constants of equivalence  depend on $N,q$ and $R$. The measure $\mu_K$ is called the capacitary measure of $K$.
\end{theorem}
\begin{theorem} \label{6h2305201410}  For any $R>0$, there exist positive constants $C_1,C_2$ such that for any $\mu\in \mathfrak{M}^+(\mathbb{R}^{N+1})$ such that $||\mathbb{M}_2^R[\mu]||_{L^\infty(\mathbb{R}^{N+1})}\le 1$, there holds
$$\frac{1}{|Q|}\int_{Q}exp(C_1\mathbb{I}^R_{2}[\chi_{Q}\mu])dxdt:=\fint_{Q}exp(C_1\mathbb{I}^R_{2}[\chi_{Q}\mu])dxdt\le C_2,$$ 
 for all $Q=\tilde{Q}_r(y,s)\subset \mathbb{R}^{N+1}$, $r>0$ , where $\chi_Q$ is the indicator function of $Q$. 
\end{theorem}
Frostman's Lemma in \cite[Th. 3.4.27]{66Tur} is at the core of the dual definition of Hausdorff capacities with doubling weight. It is easy to see that it  is valid for  the parabolic Hausdorff $\mathcal{PH}_\rho^N$-capacity version. As a consequence we have
\begin{theorem}\label{6h2305201412} There holds
\begin{align*}
 \sup\left\{\mu(K):\mu\in\mathfrak{M}^+(\mathbb{R}^{N+1}),  \operatorname{supp}(\mu)\subset K, ||\mathbb{M}_2^\rho[\mu]||_{L^\infty(\mathbb{R}^{N+1})}\leq 1\right\} \asymp \mathcal{PH}_{\rho}^N(K),
\end{align*}
 for any compact set $K\subset \mathbb{R}^{N+1}$ and $\rho>0$, where equivalent constant depends on $N$.
\end{theorem}
For our purpose, we need the some results about the behavior of the capacity with respect to dilations. 
\begin{proposition}\label{6h280520141}
Let  $K\subset \overline{\tilde{Q}_{100}(0,0)}$ be a compact set and $1<p<\frac{N+2}{2}$. Then 
\begin{align}\label{6h010620141}
\operatorname{Cap}_{2,1,p}(K)\gtrsim |K|^{1-\frac{2p}{N+2}}\,\text{ and } \,\operatorname{Cap}_{2,1,\frac{N+2}{2}}(K)\gtrsim \left(\log\left(\frac{|\tilde{Q}_{200}(0,0)|}{|K|}\right)\right)^{-\frac{N}{2}}.
\end{align}
Furthermore
\begin{align}\label{6h010620149}
& \operatorname{Cap}_{2,1,p}(K_\rho)\asymp \rho^{N+2-2p}\operatorname{Cap}_{2,1,p}(K),
\\[2mm] & \label{6h0106201410}
\frac{1}{\operatorname{Cap}_{2,1,\frac{N+2}{2}}(K_\rho)}\asymp \frac{1}{\operatorname{Cap}_{2,1,\frac{N+2}{2}}(K)}+(\log(2/\rho))^{N/2}, \end{align}
for any $0<\rho<1$, where $K_\rho= \{(\rho x,\rho^2t): (x,t)\in K\}$. 
\end{proposition}
\begin{proposition}\label{6h300520142}
Let $K\subset \overline{\tilde{Q}_1(0,0)}$ be a compact set and $1<p\leq\frac{N+2}{2}$. Then, there exists  a function $\varphi\in C_c^\infty (\tilde{Q}_{3/2}(0,0))$ with $0\leq \varphi \leq 1$ and ${\left. \varphi \right|_{D}}=1$ for some open set~$ D\supset  K$ such that 
\begin{align}\label{6h300520141}
 \int_{\mathbb{R}^{N+1}}\left(|D^2 \varphi |^p+|\nabla \varphi|^p+|\varphi|^p+|\partial_t\varphi |^p\right) dxdt\lesssim \operatorname{Cap}_{2,1,p}(K).
\end{align}
\end{proposition} 
We will give proofs of the above two propositions in the Appendix.\smallskip
 
Let $\{e^{t\Delta}\}_{t\geq 0}$ be the semigroup of contractions in $L^p$ ($1\leq p<\infty$) generated by $\Delta$.  It is wellknown that the solution  $u$ of the problem
\begin{equation}\label{6h230520146}
 \begin{array}{lll}
  \partial_tu-\Delta u =\mu \quad&\text{ in }~ \tilde{Q}_R(0,0),\\ 
  \phantom{u_t-,,\Delta }
   u =0~~&\text{ on }~~ \partial_p\tilde{Q}_R(0,0),\\ 
   \end{array} \end{equation}
 with $\mu\in C^\infty(\tilde{Q}_R(0,0))$, can be expressed by Duhamel's formula
 \begin{align*}
 u(x,t)=\int_{0}^{t}\left(e^{(t-s)\Delta}\mu\right)(x,s)ds~~\text{ for all } ~(x,t)\in \tilde{Q}_R(0,0).
 \end{align*} 
 We denote by $\mathbb{H}$ the  Gaussian kernel in $\mathbb{R}^{N+1}$: 
 $$\mathbb{H}(x,t)=\frac{1}{(4\pi t)^{\frac{N}{2}}}e^ {-\frac{|x|^2}{4t}}\chi_{t>0}.$$
We have 
 \begin{align*}
 |u(x,t)|\leq (\mathbb{H}*\mu)(x,t)~~\text{ for all } ~(x,t)\in \tilde{Q}_R(0,0).
 \end{align*}
 In \cite[Proof of Proposition 4.8]{66H1} we show that
 \begin{align*}
 |(\mathbb{H}*\mu)|(x,t) \leq C_1(N) \mathbb{I}_2^{2R}[|\mu|](x,t)~~\text{ for all } ~(x,t)\in \tilde{Q}_R(0,0).
 \end{align*}
 Here $\mu$ is extended by $0$ in $(\tilde{Q}_R(0,0))^c$. Thus, 
  \begin{align}\label{6h230520145}
  |\int_{0}^{t}\left(e^{(t-s)\Delta}\mu\right)(x,s)ds|\leq C_1(N) \mathbb{I}_2^{2R}[|\mu|](x,t)~~\text{ for all } ~(x,t)\in \tilde{Q}_R(0,0). 
  \end{align}
  Moreover, we also prove in \cite{66H1},  that if $\mu\geq 0$ then for $(x,t)\in \tilde{Q}_R(0,0) $ and $B_\rho(x)\subset B_R(0)$,
   \begin{align}\label{6h230520147}
    \int_{0}^{t}\left(e^{(t-s)\Delta}\mu\right)(x,s)ds\geq C_2(N)\sum_{k=0}^{\infty}\frac{\mu(Q_{\frac{\rho_k}{8}}(x,t-\frac{35}{128}\rho_k^2))}{\rho_k^N},
              \end{align}
    with $\rho_k=4^{-k}\rho$. \smallskip
    
It is easy to see that estimates \eqref{6h230520145} and \eqref{6h230520147} also holds for any bounded Radon measure $\mu$ in $\tilde{Q}_R(0,0)$. The following result is proved in \cite{66BaPi1} and \cite{66Ta}, and also in \cite{66H1} in a more general framework. 
%%%%%%%%%%%%%%%%%%%%%%%%%%%%%%%%%%%%%%%%%%%%%%%%%%%%%%%%%%%%%%%%%%%%%%%%%%%%%%%%%%%%%%%%%%%%%%%%%%%%%%%%%%%%%%%%%%%%
\begin{theorem}\label{6h230520148} Let $q>1$, $R>0$ and  $\mu$ be a bounded Radon measure in $\tilde{Q}_R(0,0)$. \smallskip

\noindent (i) If $\mu$ is absolutely continuous with respect to $\operatorname{Cap}_{2,1,q'}$ in $\tilde{Q}_R(0,0)$, then there exists a unique weak solution $u$ to equation
\begin{align*}
 \begin{array}{rll}
  \partial_tu-\Delta u +|u|^{q-1}u=\mu \qquad&\text{ in }~ \tilde{Q}_R(0,0),\\ 
   u =0~~~~~~&\text{ on}~~\partial_p\tilde{Q}_R(0,0).\\ 
   \end{array} \end{align*}
   
\noindent (ii)  If $ \exp\left(C_1(N) \mathbb{I}_2^{2R}[|\mu|]\right)\in L^1(\tilde{Q}_R(0,0))$, then
   there exists a unique weak solution $v$ to equation
   \begin{align*}
 \begin{array}{rll}
     \partial_tv-\Delta v +\operatorname{sign}(v)(e^{|v|}-1)=\mu \qquad&\text{ in }~ \tilde{Q}_R(0,0),\\ 
      v =0~~~~~~&\text{ on}~~\partial_p\tilde{Q}_R(0,0),\\ 
   \end{array} \end{align*}
      where the constant $C_1(N)$ is the one of inequality  \eqref{6h230520145}.    
\end{theorem} 
From estimates \eqref{6h230520145} and \eqref{6h230520147} and using comparison principle we get the estimates from below of the solutions $u$ and $v$ obtained in Theorem \ref{6h230520148}. 
\begin{proposition}\label{6h230520149}If $\mu$ is nonnegative, then  the functions $u$ and $v$ of the previous theorem are nonnegative too and satisfy
\begin{equation}\label{6hest-u}
u(x,t)\geq C_2(N)\sum_{k=0}^{\infty}\frac{\mu(Q_{\frac{\rho_k}{8}}(x,t-\frac{35}{128}\rho_k^2))}{\rho_k^N}-C_1(N)^{q+1}\mathbb{I}_2^{2R}\left[\left(\mathbb{I}_2^{2R}[\mu]\right)^q\right](x,t),
\end{equation}
and
\begin{equation}\label{6hest-v}
v(x,t)\geq C_2(N)\sum_{k=0}^{\infty}\frac{\mu(Q_{\frac{\rho_k}{8}}(x,t-\frac{35}{128}\rho_k^2))}{\rho_k^N}-C_1(N)\mathbb{I}_2^{2R}\left[\exp\left(C_1(N)\mathbb{I}_2^{2R}[\mu]\right)-1\right](x,t),
\end{equation}
 for any  $(x,t)\in \tilde{Q}_R(0,0) $ and $B_\rho(x)\subset B_R(0)$ and $\rho_k=4^{-k}\rho$.
\end{proposition}

%%%%%%%%%%%%%%%%%%%%%%%%%%%%%%%%%%%%%%%%%%%%%%%%%%%%%%%%%%%%%%%%%%%%%%%%%%%%%%%%%%%%%%%%%%%%%%%%%%%%%%%%%%%%%%%%%%%%%%%%%%%%%%%%%%%%%%%%%%%%%%%%%%%%%%%%%%%%%%%%%%%%%%%%%%%%%%%%%%%%%%%%%%%%%%%%%%%%%%%%%%%%%%%%%%%%%%%%%%%%%%%%%%%%%%%%%%%%%%%%%%%%%%%%%%%%%%%%%%%%%%%%%%%%%%%%%%%%%%%%%%%%%%%%%%

\section{Maximal solutions}
In this section we assume that $O$ is an arbitrary non-cylindrical and bounded open set in $\mathbb{R}^{N+1}$ and $q>1$. We will prove the existence of a maximal solution of 
\begin{align}\label{6h200520143}
\partial_tu-\Delta u +u^q=0 
\end{align}
in $O$. 
We also get an analogous result when $u^q$ is replaced by $e^u-1$.\smallskip

 It is easy to see that 
if $u$ satisfies \eqref{6h200520143} in $\tilde{Q}_r(0,0)~(~Q_r(0,0)~)$  then $u_a(x,t)=a^{-2/(q-1)}u(ax,a^2t)$ satisfies \eqref{6h200520143} in $\tilde{Q}_{r/a}(0,0)~(Q_{r/a}(0,0))$ for any $a>0$. 
If $X=(x,t)\in O$, the parabolic distance from $X$ to the parabolic boundary $\partial_pO$ of $O$ is defined by 
\begin{align*}
d(X,\partial_p O)=\mathop {\inf }\limits_{\scriptstyle (y,s)\in \partial_pO \hfill \atop 
  \;\;\;\;\scriptstyle s\leq t \hfill} \max\{|x-y|,(t-s)^{\frac{1}{2}}\}.
\end{align*}
It is easy to see that there exists $C=C(N,q)>0$ such that the function $V$ defined by
\begin{align*}
V(x,t)=C\left((\rho^2+t)^{-\frac{1}{q-1}}+\left(\frac{\rho^2-|x|^2}{\rho}\right)^{-\frac{2}{q-1}}\right)~\text{ in } B_\rho(0)\times (-\rho^2,0),
\end{align*} 
satisfies
\begin{align}\label{6h200520142}
\partial_tV-\Delta V+V^q\geq 0 ~~\text{ in }~B_\rho(0)\times (-\rho^2,0).
\end{align}
%%%%%%%%%%%%%%%%%%%%%%%%%%%%%%%%%%%%%%%%%%%%%%%%%%%%%%%%%%%%%%%%%%%%%%%%%%%%%%%%%%%%%%%%%%%%%%%%%%%%%%%%%%%%%%%%%%%%%%%%%%%%%%%%%%%%%%%%%%%%%%%%%%%%%%%%%%%%%%%%%%%%%%%%%%%%%%
\begin{proposition}\label{6h010620146} There exists a maximal solution $u\in C^{2,1}(O)$ of \eqref{6h200520143}
and it satisfies 
\begin{align}\label{6h200520144}
u(x,t)\leq C (d((x,t),\partial_pO))^{-\frac{2}{q-1}}~~ \text{ for all } (x,t)\in O,
\end{align}
for some $C=C(N,q)$.
\end{proposition}
\begin{proof}
Let $\mathcal{D}_k$, $k\in\mathbb{Z}$ be the collection of all the dyadic parabolic cubes (abridged $p$-cubes) of the form 
 \begin{equation*}
 \{(x_1,...,x_N,t): m_j2^{-k}\leq x_j\leq (m_j+1)2^{-k}, j=1,...,N,\, m_{N+1}4^{-k}\leq t\leq (m_{N+1}+1)4^{-k}\}
 \end{equation*}
 where $m_j\in\mathbb{Z}$. The following properties hold, 
 \begin{description}
 \item[a.] for each integer $k$, $\mathcal{D}_k$ is a partition of $\mathbb{R}^{N+1}$ and all $p$-cubes in $\mathcal{D}_k$ have the same sidelengths.
 \item[b.] if the interiors of two $p$-cubes $Q$ in $\mathcal{D}_{k_1}$ and $P$ in  $\mathcal{D}_{k_2}$, denoted $\mathop Q\limits^ \circ,\mathop P\limits^ \circ$, have nonempty intersection then either $Q$ is contained in $P$ or $Q$  contains $P$. 
 \item[c.] Each $Q$ in $\mathcal{D}_k$ is union of $2^{N+2}$ $p$-cubes in $\mathcal{D}_{k+1}$ with disjoint interiors.
 \end{description}
 Let $k_0\in\mathbb{N}$ be such that $Q\subset O$ for some $Q\in \mathcal{D}_{k_0}$. Set 
 $O_k=\bigcup\limits_{\scriptstyle Q\in\mathcal{D}_k \hfill \atop 
   \scriptstyle Q\subset O \hfill} {Q},~\forall k\geq k_0 $, 
 we have $O_k\subset O_{k+1}$ and $O=\bigcup\limits_{k\geq k_0} {O_k}=\bigcup\limits_{k\geq k_0} {\mathop O\limits^ \circ}_k$. More precisely, there exist real numbers $a_1,a_2,....,a_{n(k)}$ and open sets $\Omega_1,\Omega_2,..,\Omega_{n(k)}$ in $\mathbb{R}^N$ such that 
 $$
 a_{i}<a_{i}+4^{-k}\leq a_{i+1}<a_{i+1}+4^k~~\text{for } ~i=1,...,n(k)-1,$$
 and
 $${\mathop O\limits^ \circ}_k=\bigcup\limits_{i = 1}^{n(k)-1} \left({\Omega_i\times (a_i,a_i+4^{-k}]}\right)\bigcup \left({\Omega_{n(k)}\times (a_{n(k)},a_{n(k)}+4^{-k})}  \right).
$$
 For $k\geq k_0$, we claim that there exists a solution $u_k\in C^{2,1}({\mathop O\limits^ \circ}_k)$ to problem 
 \begin{equation}\label{6h200520141}
 \begin{array}{lll}
\partial_t u_k-\Delta u_k +u_k^q=0 \qquad&\text{ in }~ {\mathop O\limits^ \circ}_k,\\ 
  \phantom{::::---}
   u_k(x,t)\to\infty &\text{ as } d((x,t),\partial_p{\mathop O\limits^ \circ}_k)\to 0.\\ 
   \end{array} 
 \end{equation}
 Indeed, by \cite{66Di,66Li3} for $m>0$, one can find nonnegative solutions $v_i\in C^{2,1}(\Omega_i\times (a_i,a_i+4^{-k}])\cap C(\overline{\Omega}_i\times [a_i,a_i+4^{-k}])$ for $i=1,..,n(k)$ to equations 
$$
 \begin{array}{lll}
   \partial_tv_1-\Delta v_1 +v_1^q=0 \qquad&\text{ in }~ \Omega_1\times (a_1,a_1+4^{-k}),\\ 
     \phantom{-,---}
   \! v_1(x,t)=m &\text{ on }~\partial\Omega_1\times (a_1,a_1+4^{-k}),\\ 
      \phantom{----}
  \!\!  v_1(x,a_1)=m~~&\text{ in }~\Omega_1,
   \end{array} 
$$
  and 
$$
 \begin{array}{lll}
     \partial_tv_i-\Delta v_i +v_i^q=0 ~~\qquad&\text{ in }~ \Omega_i\times (a_i,a_i+4^{-k}), \\[0mm] 
         \phantom{----}
     v_i(x,t)=m ~\qquad&\text{ on } \partial\Omega_i\times (a_i,a_i+4^{-k}), \\[0mm]        \phantom{,,,--}
     v_i(x,a_i) = m_i \qquad  &\text{ in } \Omega_i,
          \end{array} 
$$
where
$$m_i=\left\{ \begin{array}{ll}
     m~\text{ in }~\Omega_i ~&\text{if }~ a_i>a_{i-1}+4^{-k},\\ 
     m\chi_{\Omega_{i}\backslash\Omega_{i-1}}(x)+v_{i-1}(x,a_{i-1}+4^{-k})\chi_{\Omega_{i-1}}(x)~&\text{otherwise  }. \\ 
     \end{array} \right. 
$$
  Clearly, 
  \begin{align*}
  u_{k,m}=v_i ~~\text{in } \Omega_i\times (a_i,a_i+4^{-k}]~~\text{ for }~i=1,...,n(k)
  \end{align*}
  is a  solution in $ C^{2,1}({\mathop O\limits^ \circ}_k)\cap C(O_k)$ to equation
  \begin{align*}
    \begin{array}{llll}
     \partial_tu_{k,m}-\Delta u_{k,m} +u_{k,m}^q=0 &\text{ in }~ {\mathop O\limits^ \circ}_k,\\ 
     \phantom{\partial_t-\Delta u_{k,m} +u_{k,m}^q}
      u_{k,m}=m &\text{ on }~\partial_p{\mathop O\limits^ \circ}_k.\\      
      \end{array} 
    \end{align*}
 Moreover, for $(x,t)\in{\mathop O\limits^ \circ}_k$, we  see that
 $B_{\frac{d}{2}}(x)\times (t-\frac{d^2}{4},t)\subset {\mathop O\limits^ \circ}_k$ where $d=d((x,t),\partial_p{\mathop O\limits^ \circ}_k)$. From \eqref{6h200520142}, we verify that 
\begin{align*}
U(y,s):=V(y-x,s-t)=C\left((\rho^2+s-t)^{-\frac{1}{q-1}}+\left(\frac{\rho^2-|x-y|^2}{\rho}\right)^{-\frac{2}{q-1}}\right)
\end{align*} 
with $\rho=d/2$, satisfies 
\begin{align}
\partial_tU-\Delta U+U^q\geq 0 ~~\text{ in }~B_{\frac{d}{2}}(x)\times (t-\frac{d^2}{4},t).
\end{align}
Applying the comparison principle we get 
\begin{align*}
u_{k,m}(y,s)\leq U(y,s)~\text{ in }~B_{\frac{d}{2}}(x)\times (t-\frac{d^2}{4},t],
\end{align*}
which implies
\begin{align}
u_{k,m}(x,t)\leq C  \left(d((x,t),\partial_p{\mathop O\limits^ \circ}_k)\right)^{-\frac{2}{q-1}}~\text{ for all }~(x,t)\in {\mathop O\limits^ \circ}_k.
\end{align}
From this, we obtain also uniform local bounds for $\{u_{k,m}\}_m$. By standard regularity theory see, \cite{66Di,66Li3},  $\{u_{k,m}\}_m$ is uniformly locally bounded  in $C^{2,1}$. Hence, up to a subsequence, $u_{k,m}\to u_{k}$  $C^{1,0}_{\text{loc}}({\mathop O\limits^ \circ}_k)$ as $m\to\infty$. We derive that $u_k$ is a weak solution of \eqref{6h200520141} in ${\mathop O\limits^ \circ}_k$, which satisfies $u_k(x,t)\to \infty$ as $d((x,t),\partial_p{\mathop O\limits^ \circ}_k)\to 0$ and 
$$u_{k}(x,t) \le C  \left(d((x,t),\partial_p{\mathop O\limits^ \circ}_k)\right)^{-\frac{2}{q-1}}~\text{ for all }~(x,t)\in {\mathop O\limits^ \circ}_k.$$
Let $m>0$ and $k\geq k_0$. Since $u_{k+1,m}\leq m$ in $O_k$ and $O_k\subset O_{k+1}$, it follows by the comparison principle applied to $u_{k+1,m}$ and $u_{k,m}$ in the following $n(k)$ sub-domains  of ${\mathop O\limits^ \circ}_k$: ${\Omega_1\times (a_1,a_1+4^{-k})}$, ${\Omega_2\times (a_2,a_2+4^{-k})}$,..., ${\Omega_{n(k)}\times (a_{n(k)},a_{n(k)}+4^{-k})}$, that $u_{k+1,m}\leq u_{k,m}$ in ${\mathop O\limits^ \circ}_k$, and thus $u_{k+1}\leq u_{k}$ in ${\mathop O\limits^ \circ}_k$ by letting $m\to\infty$. In particular, $\{u_{k}\}_k$ is uniformly locally bounded in $L^\infty_{\text{loc}}$. We use the same compactness property as above to infer  that $u_k\to u$ as $k\to\infty$. Then $u$ is a solution of \eqref{6h200520143} and it satisfies  \eqref{6h200520144}. By construction $u$ is the maximal solution. 
\end{proof}
%%%%%%%%%%%%%%%%%%%%%%%%%%%%%%%%%%%%%%%%%%%%%%%%%%%%%%%%%%%%%%%%%%%%%%%%REMARK%%%%%%%%%%%%%%%%%%%%%%%%%%%%%%%%%%%%%%%%%%%%%%%%%%%%%%%%%%%%%%%%%%%%%%%%%%%%%%%%%%%%%%%%%%%%%%%%%
\begin{remark}\label{6h010620147}Let $R\geq 2 r\geq 2$,  $K$ be a compact subset in $\overline{\tilde{Q}_r(0,0)}$. As in the proof of Proposition \ref{6h010620146},  we can show that there exists a maximal solution of 
 \begin{equation}\label{6h290520141}
 \begin{array}{lll}
   \partial_tu-\Delta u +u^q=0\qquad&\text{ in }~~ \tilde{Q}_R(0,0)\backslash K,\\ 
   \phantom{   \partial_tu-\Delta +u^q}
    u =0~~&\text{ on }~~ \partial_p\tilde{Q}_R(0,0),\\ 
    \end{array} \end{equation}
    which satisfies 
    \begin{align}
    \label{6h310520145}u(x,t)\leq C (d((x,t),\partial_p(\tilde{Q}_R(0,0)\backslash K))^{-\frac{2}{q-1}}~~\forall ~(x,t)\in  \tilde{Q}_{R}(0,0)\backslash K, 
    \end{align}
    for some $C=C(N,q)$.
    Furthermore, assume $K_1,K_2,,,,K_m$ are compact subsets in $\overline{\tilde{Q}_r(0,0)}$ and $K=K_1\cup...\cup K_m$. Let $u,u_1,...,u_m$ be the maximal solutions of \eqref{6h290520141} in  $\tilde{Q}_R(0,0)\backslash K,$ $\tilde{Q}_R(0,0)\backslash K_1,$ $\tilde{Q}_R(0,0)\backslash K_2,,,,\tilde{Q}_R(0,0)\backslash K_m $, respectively, then 
    \begin{align}\label{6h010620148}
    u\leq \sum_{j=1}^{m}u_j~~\text{in }~\tilde{Q}_R(0,0)\backslash K.
    \end{align} 
\end{remark}
\begin{remark}\label{6h260520147}
If the equation  \eqref{6h200520143} admits a large solution for some $q>1$, then  for any $1<q_1<q$, the equation
\begin{align}\label{6h210520148}
\partial_tu-\Delta u +u^{q_1}=0 \text{ in }~O
\end{align}
admits also a large solution.\\
Indeed, assume that $u$ is a large solution of \eqref{6h200520143} and $v$ the maximal solution of \eqref{6h210520148}. Take $R>0$ such that $O\subset B_R(0)\times (-R^2,R^2)$,  then
the function $V$ defined by 
$$V(x,t)=(q-1)^{-\frac{1}{q-1}}(2R^2+t)^{-\frac{1}{q-1}},$$ 
satisfies \eqref{6h200520143}. It follows for all $(x,t)\in O$
$$
u(x,t)\geq \inf_{(y,s)\in O}V(x,t)\geq (q-1)^{-\frac{1}{q-1}}R^{-\frac{2}{q-1}}=:a_0.$$
Then $\tilde{u}=a_0^{\frac{q-q_1}{q_1-1}}u$ is a subsolution of \eqref{6h210520148}. Therefore $v\geq a_0^{\frac{q-q_1}{q_1-1}}u$ in $O$, thus $v$ is a large solution. 
\end{remark}
\begin{remark}[Sub-critical case]\label{6h260520143}Assume that $1<q<q_*$.  It is easy to check that 
the function \begin{align}\label{6h260520142}
U(x,t)= \frac{C}{t^{\frac{1}{q-1}}}e^{-\frac{|x|^2}{4t}}\chi_{t>0}
\end{align}
is a subsolution of \eqref{6h200520143} in $\mathbb{R}^{N+1}\backslash\{(0,0)\}$, where $C=\left(\frac{2}{q-1}-\frac{N}{2}\right)^{\frac{1}{q-1}}$.\\
Therefore, the maximal solution $u$ of \eqref{6h200520143} in $O$ verifies
\begin{align}\label{6h200620142}
u(x,t)\geq C \frac{1}{(t-s)^{\frac{1}{q-1}}}e^{-\frac{|x-y|^2}{4(t-s)}}\chi_{t>s}, 
\end{align}
for all $(x,t)\in O$ and $(y,s)\in O^c$. \\
If for any $(x,t)\in\partial_pO$ there exist $\varepsilon\in (0,1)$ and a decreasing sequence $\{\delta_n\}\subset (0,1)$ converging to $0$ as $n\to\infty$ such that $\left(B_{\delta_n}(x)\times (-\delta_n^2+t,-\varepsilon\delta_n^2+t)\right)\cap O^c
\not= \emptyset$ for any $n\in\mathbb{N}$, then $u$ is a large solution. For proving this,  we need to show that 
$$\lim\limits_{\rho\to 0}\inf_{O\cap (B_\rho(x)\times (-\rho^2+t,\rho^2+t))}u =\infty.$$  
Let $0<\rho<\sqrt{\frac{\varepsilon}{2}}\delta_1$ and $n\in\mathbb{N}$ such that  $\sqrt{\frac{\varepsilon}{2}}\delta_{n+1}\leq\rho<\sqrt{\frac{\varepsilon}{2}}\delta_n$. \\Since $\left(B_{\delta_n}(x)\times (-\delta_n^2+t,-\varepsilon\delta_n^2+t)\right)\cap O^c
\not= \emptyset$, there is $(x_n,t_n)\in O^c$ such that $|x_n-x|<\delta_n$ and $-\delta_n^2+t<t_n<-\varepsilon\delta_n^2+t$. So if  $(y,s)\in O\cap (B_\rho(x)\times (-\rho^2+t,\rho^2+t))$ then $|y-x_n|<(\sqrt{\varepsilon}+1)\delta_{n}$ and $\frac{\varepsilon}{2}\delta_n^2<s-t_n<(\varepsilon+1)\delta_n^2$. Hence, thanks to \eqref{6h200620142} we have for any  $(y,s)\in O\cap (B_\rho(x)\times (-\rho^2+t,\rho^2+t))$ 
\begin{align*}
u(y,s)\geq C \frac{1}{(s-t_n)^{\frac{1}{q-1}}}e^{-\frac{|y-x_n|^2}{4(s-t_n)}}\geq C (\varepsilon+1)^{-\frac{1}{q-1}}e^{-\frac{(\sqrt{\varepsilon}+1)^2}{2\varepsilon}}\delta_n^{-\frac{2}{q-1}},
\end{align*}
which implies
\begin{align*}
\inf_{O\cap (B_\rho(x)\times (-\rho^2+t,\rho^2+t))}u\geq C (\varepsilon+1)^{-\frac{1}{q-1}}e^{-\frac{(\sqrt{\varepsilon}+1)^2}{2\varepsilon}}\delta_n^{-\frac{2}{q-1}}\to \infty ~\text{ as }
~ \rho\to 0.\end{align*}
 \end{remark}
%%%%%%%%%%%%%%%%%%%%%%%%%%%%%%%%%%%%%%%%%%%%%%%%%%%%%%%%%%%%%%%%%%%%%%%%%%%%%%%%%%%%%%%%%%%%%%%%%%%%%%%%%%%%%%%%%%%%
\begin{remark}\label{6h260520146}Note that if $u\in C^{2,1}(O)$ is a solution of \eqref{6h200520143} for some $q>1$ then, for 
$a,b>0$ and $1<p\leq 2$, the function $v=b^{-\frac{1}{q-1}}u$ is a super-solution of 
\begin{align}\label{6h260520141}
\partial_tv-\Delta v+a|\nabla v|^p+ bv^q= 0~~\text{ in }~O.
\end{align}
Thus, we can apply the argument of the previous proof, with equation  \eqref{6h200520143} replaced by \eqref{6h260520141}, and deduce that there exists a maximal solution $v\in C^{2,1}(O)$ of \eqref{6h260520141} satisfying 
\begin{align*}
v(x,t)\leq C b^{-\frac{1}{q-1}} (d((x,t),\partial_pO))^{-\frac{2}{q-1}}~~ \text{ for all } (x,t)\in O.
\end{align*}
Furthermore, if $1<q<q_*$,  $q=\frac{2p}{p+1}$, $a,b>0$ then the function $U$ expressed by (\ref{6h260520142}) in Remark \ref{6h260520143} is a subsolution of \eqref{6h260520141} in $\mathbb{R}^{N+1}\backslash\{(0,0)\}$, provided the explicit constant $C$ given therein is replaced by some $C=C(N,p,q,a,b)$. Therefore, we conclude that every maximal solution of $v\in C^{2,1}(O)$ of \eqref{6h260520141} satisfy 
\begin{align}
v(x,t)\geq C \frac{1}{(t-s)^{\frac{1}{q-1}}}e^{-\frac{|x-y|^2}{4(t-s)}}\chi_{t>s} 
\end{align}
for all $(x,t)\in O$ and $(y,s)\in\partial_p O$.\smallskip

\noindent Arguing as in Remark \ref{6h260520143}, if for any $(x,t)\in\partial_pO$ there exist $\varepsilon\in (0,1)$ and a decreasing sequence $\{\delta_n\}\subset (0,1)$ converging to $0$ as $n\to\infty$ such that $\left(B_{\delta_n}(x)\times (-\delta_n^2+t,-\varepsilon\delta_n^2+t)\right)\cap O^c
\not= \emptyset$ for any $n\in\mathbb{N}$, then $v$ is a large solution.
\end{remark}

Next, we consider the following equation
\begin{align}
 \label{6h240520142}\partial_tu-\Delta u+e^u-1=0.
\end{align}  
It is easy to see that the two functions
\begin{align*}
& V_1(t)=-\log\left(\frac{t+\rho^2}{1+\rho^2}\right) ~\text{ and }~V_2(x)=C-2\log\left(\frac{\rho^2-|x|^2}{\rho}\right)
\end{align*}
satisfy 
$$
V'_1+e^{V_1}-1\geq 0 \qquad\text{ in }~ (-\rho^2,0],$$
and
$$
-\Delta V_2+e^{V_2}-1\geq 0\qquad\text{ in }~ B_\rho(0),
$$
for some $C=C(N)$. 
Using $e^{a}+e^{b}\leq e^{a+b}-1$ for $a,b\geq 0$, we obtain that $V_1+V_2$ is a supersolution of equation \eqref{6h240520142} in  $B_\rho(0)\times (-\rho^2,0]$.
By the same argument as in Proposition \ref{6h010620146} and the estimate of the above supersolution, we infer the following:
%%%%%%%%%%%%%%%%%%%%%%%%%%%%%%%%%%%%%%%%%%%%%%%%%%%%%%%%%%%%%%%%%%%%%%%%%%%%%%%%%%%%%%%%%%%%%%%%%%%%%%%%%%%%%%%%%%%%%%%%%%%%%%%%%%%%%%%%%%%%%%%%%%%%%%%%%%%%%%%%%%%%%%%%%%%%%%
\begin{proposition}\label{6h01062014e} There exists a maximal solution $u\in C^{2,1}(O)$ of 
\begin{align}\label{6h200520145}
\partial_tu-\Delta u +e^u-1=0 \text{ in }~O,
\end{align}
and it satisfies 
\begin{align}\label{6h200520146}
u(x,t)\leq C-\log\left(\frac{(d((x,t),\partial_pO))^3}{4+(d((x,t),\partial_pO))^2}\right)~~ \text{ for all } (x,t)\in O,
\end{align}
 for some $C=C(N)$. 
\end{proposition}
The next three propositions will be useful to prove Theorem \ref{6h230520143}-(ii). 
%%%%%%%%%%%%%%%%%%%%%%%%%%%%%%%%%%%%%%%%%%%%%%%%%%%%%%%%%%%%%%%%%%%%%%%%%%%%%%%%%%%%%%%%%%%%%%%%%%%%%%%%%%%%%%%%%%%%%%%%%%%%%%%%%%%%%%%%%%%%%%%%%%%%%%%%%%%%%%%%%%%%%%%%%%%%%%

\begin{proposition}\label{6h010620143}Let $K\subset \overline{\tilde{Q}_1(0,0)}$ be a compact set and $q>1$, $R\geq 100$. Let  $u$ be a solution of  \eqref{6h290520141} in $\tilde{Q}_R(0,0)\backslash K$ and $\varphi$ as in Proposition \ref{6h300520142} with $p=q'$. Set $\xi=(1-\varphi)^{2q'}$. Then, 
\begin{align}\label{6h310520143}
\int_{\tilde{Q}_R(0,0)}u\left(|\Delta \xi|+|\nabla \xi|+|\partial_t\xi|\right)dxdt\lesssim \operatorname{Cap}_{2,1,q'}(K),
\end{align}

 \begin{align}\label{6h310520144}
 u(x,t)\lesssim \operatorname{Cap}_{2,1,q'}(K) +R^{-\frac{2}{q-1}}~~\text{for any }~ (x,t)\in \tilde{Q}_{R/5}(0,0)\backslash \tilde{Q}_2(0,0),
 \end{align}
  and 
  \begin{align}\label{6h310520146}
  \int_{\tilde{Q}_2(0,0)}u\xi dx dt\lesssim \operatorname{Cap}_{2,1,q'}(K) +R^{-\frac{2}{q-1}},
  \end{align}
 where the constants in above inequalities depend only on $N$ and $q$.  
\end{proposition}
\begin{proof} {\it Step 1.}
We claim that 
\begin{align}\label{6h310520141}
\int_{\tilde{Q}_R(0,0)}u^q\xi dxdt\lesssim \text{Cap}_{2,1,q'}(K).
\end{align}
Actually, using integration by parts and the Green formula, one has 
\begin{align*}
&\int_{\tilde{Q}_R(0,0)}u^q\xi dxdt=-\int_{\tilde{Q}_R(0,0)} \partial_tu\xi dxdt+\int_{\tilde{Q}_R(0,0)} \xi\Delta u  dxdt\\
&\phantom{------} = \int_{\tilde{Q}_R(0,0)} u\partial_t\xi dxdt+ \int_{\tilde{Q}_R(0,0)} u\Delta\xi   dxdt +\int_{-R^2}^{R^2}\int_{\partial B_R(0)}\left(\xi\frac{\partial u}{\partial \nu}- u\frac{\partial\xi}{\partial \nu}\right)dSdt,
\end{align*}
where $\nu$ is the outer normal unit vector on $\partial B_R(0)$. Clearly,  
\begin{align*}
\frac{\partial u}{\partial \nu}\leq 0 ~~\text{ and }~~\frac{\partial\xi}{\partial \nu}=0~~\text{ on }~~\partial B_R(0).
\end{align*}
Thus, 
\begin{align}\nonumber
&\int_{\tilde{Q}_R(0,0)}u^q\xi dxdt\leq \int_{\tilde{Q}_R(0,0)} u|\partial_t\xi| dxdt+ \int_{\tilde{Q}_R(0,0)} u|\Delta\xi|   dxdt 
\\&\phantom{---}\nonumber\leq 2q'\int_{\tilde{Q}_R(0,0)} u(1-\varphi)^{2q'-1}|\partial_t\varphi| dxdt+2q'(2q'-1)\int_{\tilde{Q}_R(0,0)}u(1-\varphi)^{2q'-2}|\nabla\varphi|^2dxdt\\&\nonumber~~~~~~~~~~~+2q'\int_{\tilde{Q}_R(0,0)}u(1-\varphi)^{2q'-1}|\Delta \varphi|dxdt
\\&\phantom{---}\nonumber\leq 2q'\int_{\tilde{Q}_R(0,0)} u\xi^{1/q}|\partial_t\varphi| dxdt+2q'(2q'-1)\int_{\tilde{Q}_R(0,0)}u\xi^{1/q}|\nabla\varphi|^2dxdt\\&~~~~~~~~~~~+2q'\int_{\tilde{Q}_R(0,0)}u\xi^{1/q}|\Delta \varphi|dxdt.\label{6h310520142}
\end{align}
In the last inequality, we have used the fact that $(1-\phi)^{2q'-1}\leq (1-\phi)^{2q'-2}=\xi^{1/q}$. \\
Hence, by H\"older's inequality, 
\begin{align*}
&\int_{\tilde{Q}_R(0,0)}u^q\xi dxdt
\lesssim \int_{\tilde{Q}_R(0,0)}| \partial_t \varphi|^{q'} dxdt+\int_{\tilde{Q}_R(0,0)}|\nabla\varphi|^{2q'}dxdt\\&~~~~~~~~~~~~~~~~~~~~~~+\int_{\tilde{Q}_R(0,0)}|\Delta \varphi|^{q'}dxdt.
\end{align*}
By the Gagliardo-Nirenberg inequality, 
\begin{align*}
\int_{\tilde{Q}_R(0,0)}|\nabla\varphi|^{2q'}dxdt&\lesssim ||\varphi||_{L^\infty(\tilde{Q}_R(0,0))}^{q'} \int_{\tilde{Q}_R(0,0)}|D^2\varphi|^{q'}dxdt\\&
\lesssim \int_{\tilde{Q}_R(0,0)}|D^2\varphi|^{q'}dxdt.
\end{align*}
Hence, we find
\begin{align*}
&\int_{\tilde{Q}_R(0,0)}u^q\xi dxdt
\lesssim \int_{\tilde{Q}_R(0,0)}(|\partial_t\varphi|^{q'}+|D^2 \varphi|^{q'})dxdt,
\end{align*} 
and derive \eqref{6h310520141} from \eqref{6h300520141}. In view of \eqref{6h310520142}, we also obtain 
\begin{align*}
&\int_{\tilde{Q}_R(0,0)}u(|\Delta \xi|+|\partial_t\xi|) dxdt
\lesssim \text{Cap}_{2,1,q'}(K),
\end{align*}
and
\begin{align*}
&\int_{\tilde{Q}_R(0,0)}u|\nabla \xi|dxdt
\lesssim \text{Cap}_{2,1,q'}(K),
\end{align*}
since
\begin{align*}
\int_{\tilde{Q}_R(0,0)}u|\nabla\xi| dxdt&= 2q'\int_{\tilde{Q}_R(0,0)}u\xi^{(2q'-1)/2q'}|\nabla\varphi| dxdt\\&\leq 2q'\int_{\tilde{Q}_R(0,0)}u\xi^{1/q}|\nabla\varphi| dxdt
\\& \lesssim \int_{\tilde{Q}_R(0,0)}u^q\xi dxdt+\int_{\tilde{Q}_R(0,0)}|\nabla\varphi|^{q'} dxdt.
\end{align*}
It yields \eqref{6h310520143}. \\
%%%%%%%%%%%%%%%%%%%%%%%%%%%%%%%%%%%%%%%%%%%%%%%%%%%%%%%%%%%%%%%STEP%%2%%%%%%%%%%%%%%%%%%%%%%%%%%%%%%%%%%%%%%%%%%%%%%%%%%%%%%%%%%%%%%%%%%%%%%%%%%%%%%%%%%%%%%%%%%%%%%%%%%%%%%%%%%%%%%%%%%%%%%%%%%%%%%%%%%%%%%%%%%%%%%%%%%%
 {\it Step 2.} Relation \eqref{6h310520144} holds. Let $\eta$ be a cut off function on $\tilde{Q}_{R/4}(0,0)$ with respect to $\tilde{Q}_{R/3}(0,0)$ such that $|\partial_t\eta|+|D^2\eta|\lesssim R^{-2}$ and $|\nabla \eta|\lesssim R^{-1}$. We have 
\begin{align*}
\partial_t(\eta\xi u)-\Delta (\eta\xi u)=F\in C_c(\tilde{Q}_{R/3}(0,0)).
\end{align*} 
Hence, we can write 
\begin{align*}
(\eta\xi u)(x,t)=\int_{\mathbb{R}^{N}}\int_{-\infty}^{t}\frac{1}{(4\pi (t-s))^{\frac{N}{2}}}e^ {-\frac{|x-y|^2}{4(t-s)}}F(y,s)dsdy~~
\forall (x,t)\in\mathbb{R}^{N+1}.
\end{align*}
Now, we fix $(x,t)\in \tilde{Q}_{R/5}(0,0)\backslash \tilde{Q}_2(0,0)$.  Since $\text{supp}\{|\nabla\eta|\}\cap \text{supp}\{|\nabla\xi|\}=\emptyset$ and 
\begin{align*}
F&=\eta\xi\left(\partial_tu-\Delta u\right)-2\left(\eta\nabla\xi+\xi\nabla\eta\right)\nabla u+\left(\xi\partial_t\eta+\eta\partial_t\xi-2\nabla \eta\nabla \xi-\Delta\eta \xi-\eta\Delta\xi\right)u\\[2mm]&
\leq -2\left(\eta\nabla\xi+\xi\nabla\eta\right)\nabla u+\left(\xi\partial_t\eta+\eta\partial_t\xi-\xi\Delta\eta -\eta\Delta\xi\right)u,
\end{align*}
there holds
\begin{align*}
u(x,t) =(\eta\xi u)(x,t)&\leq -2\int_{\mathbb{R}^{N}}\int_{-\infty}^{t}\frac{1}{(4\pi (t-s))^{\frac{N}{2}}}e^ {-\frac{|x-y|^2}{4(t-s)}}\left(\eta\nabla\xi+\xi\nabla\eta\right)\nabla udsdy\\&\phantom{--} +
\int_{\mathbb{R}^{N}}\int_{-\infty}^{t}\frac{1}{(4\pi (t-s))^{\frac{N}{2}}}e^ {-\frac{|x-y|^2}{4(t-s)}}\left(\eta\partial_t\xi-\eta\Delta\xi\right)udsdy
\\&\phantom{--}+\int_{\mathbb{R}^{N}}\int_{-\infty}^{t}\frac{1}{(4\pi (t-s))^{\frac{N}{2}}}e^ {-\frac{|x-y|^2}{4(t-s)}}\left(\partial_t\eta\xi-\xi\Delta\eta\right)udsdy.\\\;& =I_1+I_2+I_3.
\end{align*}
By integration  by parts,
\begin{align*}
I_1&=2(4\pi)^{-N/2}\int_{-\infty}^{t}\int_{\mathbb{R}^N}\frac{(x-y)}{2(t-s)^{(N+2)/2}}e^ {-\frac{|x-y|^2}{4(t-s)}}\left(\eta \nabla \xi+\xi \nabla \eta\right)u dyds\\& +
2(4\pi)^{-N/2}\int_{-\infty}^{t}\int_{\mathbb{R}^N}\frac{1}{(t-s)^{N/2}}e^ {-\frac{|x-y|^2}{4(t-s)}}\left(\xi \Delta \eta+\eta \Delta\xi\right)u\,dyds.
\end{align*}
Note that 
$$
\frac{1}{(t-s)^{N/2}}e^ {-\frac{|x-y|^2}{4(t-s)}}\lesssim \left(\max\{|x-y|,|t-s|^{1/2}\}\right)^{-N},
$$
$$
\left|\frac{(x-y)}{2(t-s)^{(N+2)/2}}e^ {-\frac{|x-y|^2}{4(t-s)}}\right|\lesssim \left(\max\{|x-y|,|t-s|^{1/2}\}\right)^{-N-1},
$$
and 
\begin{align*}
&\max\{|x-y|,|t-s|^{1/2}\}\gtrsim 1 ~~~~\forall (y,s)\in \text{supp}\{|D^\alpha \xi|\}\cup\text{supp}\{|\partial_t\xi|\},\\& 
\max\{|x-y|,|t-s|^{1/2}\}\gtrsim R~~~~\forall (y,s)\in \text{supp}\{|D^\alpha \eta|\}\cup\text{supp}\{|\partial_t\eta|\}~~\forall |\alpha|\geq 1.
\end{align*}
We deduce 
\begin{align*}
I_1&\lesssim \int_{\mathbb{R}^{N+1}}\left(\max\{|x-y|,|t-s|^{1/2}\}\right)^{-N-1}(\eta|\nabla \xi|+\xi|\nabla\eta|)u\,dyds\\&\phantom{--}
+\int_{\mathbb{R}^{N+1}}\left(\max\{|x-y|,|t-s|^{1/2}\}\right)^{-N}\left(\xi |\Delta \eta|+\eta| \Delta\xi|\right)u\,dyds\\&\lesssim 
\int_{\mathbb{R}^{N+1}}(|\nabla\xi|+|\Delta\xi|)u\,dyds+ \int_{\tilde{Q}_{R/3}(0,0)\backslash\tilde{Q}_{R/4}(0,0)}(R^{-N-1}|\nabla\eta|+R^{-N}|\Delta\eta|)u\,dyds
\\&\lesssim 
\int_{\mathbb{R}^{N+1}}(|\nabla\xi|+|\Delta\xi|)u\,dyds+ \sup_{\tilde{Q}_{R/3}(0,0)\backslash\tilde{Q}_{R/4}(0,0)}u,
\end{align*}
\begin{align*}
I_2&\lesssim \int_{\mathbb{R}^{N+1}}\left(\max\{|x-y|,|t-s|^{1/2}\}\right)^{-N}(|\partial_t\xi|+|\Delta\xi|)u\,dyds\\&
\lesssim \int_{\mathbb{R}^{N+1}}(|\partial_t\xi|+|\Delta\xi|)u\,dyds,
\end{align*}
and
\begin{align*}
I_3&\lesssim \int_{\mathbb{R}^{N+1}}\left(\max\{|x-y|,|t-s|^{1/2}\}\right)^{-N}(|\partial_t\eta|+|\Delta\eta|)u\,dyds\\&
\lesssim \int_{\tilde{Q}_{R/3}(0,0)\backslash\tilde{Q}_{R/4}(0,0)}R^{-N}(|\partial_t\eta|+|\Delta\eta|)u\,dyds
\\&\lesssim \sup_{\tilde{Q}_{R/3}(0,0)\backslash\tilde{Q}_{R/4}(0,0)}u.
\end{align*}
Hence, 
\begin{align*}
u(x,t)\leq I_1+I_2+I_3\lesssim \int_{\mathbb{R}^{N+1}}(|\partial_t\xi|+|\nabla\xi|+|\Delta\xi|)u\,dyds+\sup_{\tilde{Q}_{R/3}(0,0)\backslash\tilde{Q}_{R/4}(0,0)}u.
\end{align*}
Combining this inequality with \eqref{6h310520143} and \eqref{6h310520145}, we obtain \eqref{6h310520144}.\\
%%%%%%%%%%%%%%%%%%%%%%%%%%%%%%%%%%%%%%%%%%%%%%%%%%%%%%%%%%%%%%%STEP%%3%%%%%%%%%%%%%%%%%%%%%%%%%%%%%%%%%%%%%%%%%%%%%%%%%%%%%%%%%%%%%%%%%%%%%%%%%%%%%%%%%%%%%%%%%%%%%%%%%%%%%%%%%%%%%%%%%%%%%%%%%%%%%%%%%%%%%%%%%%%%%%%%%%%
 \noindent {\it Step 3.} End of the proof. Let $\theta$ be a cut off function on $\tilde{Q}_{3}(0,0)$ with respect to $\tilde{Q}_{4}(0,0)$. As above, we have for any $(x,t)\in\mathbb{R}^{N+1}$
\begin{align*}
(\theta\xi u)(x,t)&\lesssim \int_{\mathbb{R}^{N+1}}(\max\{|x-y|,|t-s|^{1/2}\})^{-N-1}(\theta|\nabla \xi|+\xi|\nabla\theta|)u\,dyds
\\&\phantom{---} +\int_{\mathbb{R}^{N+1}}(\max\{|x-y|,|t-s|^{1/2}\})^{-N}(\theta|\Delta \xi|+\xi|\Delta\theta|)u\,dyds
\\&\phantom{---} +\int_{\mathbb{R}^{N+1}}(\max\{|x-y|,|t-s|^{1/2}\})^{-N}(\theta|\partial_t\xi|+\theta|\Delta\xi|)u\,dyds
\\&\phantom{---} +\int_{\mathbb{R}^{N+1}}(\max\{|x-y|,|t-s|^{1/2}\})^{-N}(\xi|\partial_t\theta|+\xi|\Delta\theta|)u\,dyds.
\end{align*}
Hence, by Fubini theorem, 
\begin{align*}
\int_{\tilde{Q}_2(0,0)}\eta u dxdt&=\int_{\tilde{Q}_2(0,0)}\theta\eta u dxdt\\&\lesssim A\int_{\mathbb{R}^{N+1}}\left(\theta|\nabla \xi|+\xi|\nabla\theta|+\theta|\Delta \xi|+\xi|\Delta\theta|+\theta|\partial_t\xi|+\xi|\partial_t\theta|\right)u\,dyds
\\&\lesssim \int_{\mathbb{R}^{N+1}}(|\partial_t\xi|+|\nabla\xi|+|\Delta\xi|)u\,dyds+\sup_{\tilde{Q}_{4}(0,0)\backslash\tilde{Q}_{3}(0,0)}u,
\end{align*}
where \begin{align*}
A=\sup_{(y,s)\in\tilde{Q}_4(0,0) }\int_{\tilde{Q}_2(0,0)}((\max\{|x-y|,|t-s|^{1/2}\})^{-N}+(\max\{|x-y|,|t-s|^{1/2}\})^{-N-1})dxdt.
\end{align*}
Therefore we obtain \eqref{6h310520146} from \eqref{6h310520143} and \eqref{6h310520144}. 
\end{proof}
%%%%%%%%%%%%%%%%%%%%%%%%%%%%%%%%%%%%%%%%%%%%%%%%%%%%%%%%%%%%%%%%%%%%%%%%%%%%%%%%%%%%%%%%%%%%%%%%%%%%%%%%%%%%%%%%%%%%%%%%%%%%%%%%%%%%%%%%%%%%%%%%%%%%%%%%%%%%%%%%%%%%%%%%%%%%%
\begin{proposition}\label{6h0106201414}Let $0<\varepsilon<1$, $K\subset \{(x,t):\varepsilon <\max\{|x|,|t|^{1/2}\}<1\}$ be a compact set  and $u$ the maximal solution of  \eqref{6h290520141} in $\tilde{Q}_R(0,0)\backslash K$ with $R\geq 100$. Then 
\begin{align}\label{6h0106201411}
\sup_{\tilde{Q}_{\varepsilon/4}(0,0)}u\lesssim \sum_{j=-2}^{j_{\varepsilon}-2}\frac{\operatorname{Cap}_{2,1,q'}(K\cap \tilde{Q}_{\rho_j}(0,0))}{\rho_j^N}+j_\varepsilon R^{-\frac{2}{q-1}}~~\text{if }~~q>q_*,
\end{align}
and 
\begin{align}\label{6h0106201412}
\sup_{\tilde{Q}_{\varepsilon/4}(0,0)}u\lesssim \sum_{j=0}^{j_{\varepsilon}}\frac{\operatorname{Cap}_{2,1,q'}(K_j)}{\rho_j^N}+j_\varepsilon R^{-\frac{2}{q-1}}~~\text{if }~~q=q_*,
\end{align}
where $\rho_j=2^{-j}$, $K_j=\{(x/\rho_{j+3},t/\rho_{j+3}^2): (x,t)\in K\cap \tilde{Q}_{\rho_{j-2}}(0,0)  \}$ and $j_\varepsilon\in\mathbb{N}$ is such that $\rho_{j_\varepsilon}\leq \varepsilon <\rho_{j_\varepsilon-1}$. 
\end{proposition}
\begin{proof}
For 
$j\in N$, we define $S_j=\{x:\rho_j\leq\max\{|x|,|t|^{1/2}\}\leq \rho_{j-1}\}.$\\
Fix any $1\leq j\leq j_\varepsilon$. We cover $S_j$ by $L=L(N)\in\mathbb{N}^*$ closed cylinders
$$\overline{\tilde{Q}_{\rho_{j+3}}(x_{k,j},t_{k,j})},~~k=1,...,L(N),$$
where $(x_{k,j},t_{k,j})\in S_j$. \\
For $k=1,...,L(N)$, let $u_j, u_{k,j}$ be the maximal solutions of  \eqref{6h290520141} where $K$ is replaced by  $K\cap S_j$ and $K\cap \overline{\tilde{Q}_{\rho_{j+3}}(x_{k,j},t_{k,j})}$, respectively. Clearly the function $\tilde{u}_{k,j}$ defined by 
 $$\tilde{u}_{k,j}(x,t)=\rho_{j+3}^{\frac{2}{q-1}}u_{k,j}(\rho_{j+3} x+x_{k,j},\rho_{j+3}^2t+t_{k,j})$$
 is the maximal solution of  \eqref{6h290520141} provided $(K_{k,j}, \tilde{Q}_{R/\rho_{j+3}}(-x_{k,j}/\rho_{j+3},-t_{k,j}/\rho_{j+3}^2))$ with
\begin{align*}
K_{k,j}=\{(y/\rho_{j+3},s/\rho_{j+3}^2): (y,s)\in -(x_{k,j},t_{k,j})+K\cap \overline{\tilde{Q}_{\rho_{j+3}}(x_{k,j},t_{k,j})} \}\subset \overline{\tilde{Q}_{1}(0,0)}
\end{align*}
 is replacing   $(K, \tilde{Q}_R(0,0))$. Let $\overline{u}_{k,j}$ be the maximal solution of  \eqref{6h290520141} with $(K, \tilde{Q}_R(0,0))$ replaced by  $(K_{k,j}, \tilde{Q}_{2R/\rho_{j+3}}(0,0))$. 
Since $ \tilde{Q}_{R/\rho_{j+3}}(-x_{k,j}/\rho_{j+3},-t_{k,j}/\rho_{j+3}^2)\subset  \tilde{Q}_{2R/\rho_{j+3}}(0,0)$, then, by the comparison principle as in the proof of Proposition \ref{6h010620146}, we get $\tilde{u}_{k,j}\leq \overline{u}_{k,j}$ in $\tilde{Q}_{R/\rho_{j+3}}(-x_{k,j}/\rho_{j+3},-t_{k,j}/\rho_{j+3}^2)\backslash K_{k,j}$, and thus 
 \begin{align*} \tilde{u}_{k,j}(x,t)\lesssim \text{Cap}_{2,1,q'}(K_{k,j}) +(R/\rho_{j+3})^{-\frac{2}{q-1}},
 \end{align*}
 for any $(x,t)\in  \left(\tilde{Q}_{2R/(5\rho_{j+3})}(0,0)\cap\tilde{Q}_{R/\rho_{j+3}}(-x_{k,j}/\rho_{j+3},-t_{k,j}/\rho_{j+3}^2)\right)\backslash \tilde{Q}_2(0,0)=D$. \\
Fix $(x_0,t_0)\in \tilde{Q}_{\varepsilon/4}(0,0)$. Clearly, $((x_0-x_{k,j})/\rho_{j+3},(t_0-t_{k,j})/\rho_{j+3})\in D$, hence 
$$\begin{array} {lll}
u_{k,j}(x_0,t_0)=\rho_{j+3}^{-\frac{2}{q-1}} \tilde{u}_{k,j}((x_0-x_{k,j})/\rho_{j+3},(t_0-t_{k,j})/\rho_{j+3}^2)\\
[2mm]
\phantom{u_{k,j}(x_0,t_0)}
\lesssim \displaystyle\frac{\text{Cap}_{2,1,q'}(K_{k,j})}{\rho_{j}^{\frac{2}{q-1}}} +R^{-\frac{2}{q-1}}. 
\end{array}$$
Therefore, using \eqref{6h010620148} in Remark \ref{6h010620147} and the fact that 
$$
\text{Cap}_{2,1,q'}(K_{k,j})=\text{Cap}_{2,1,q'}(K_{k,j}+(x_{k,j}/\rho_{j+3},t_{k,j}/\rho_{j+3}^2))\leq \text{Cap}_{2,1,q'}(K_{j}),
$$ 
we derive  
\begin{align*}
u(x_0,t_0) \leq \sum_{j=1}^{j_\varepsilon}u_j(x_0,t_0)&\leq \sum_{j=1}^{j_\varepsilon}\sum_{k=1}^{L(N)}u_{k,j}(x_0,t_0)\\&\lesssim \sum_{j=0}^{j_\varepsilon} \frac{\text{Cap}_{2,1,q'}(K_{j})}{\rho_{j}^{\frac{2}{q-1}}} +j_\varepsilon R^{-\frac{2}{q-1}},
\end{align*}
which yields \eqref{6h0106201412}. 
If $q>q_*$, then by \eqref{6h010620149} in Proposition \ref{6h280520141}, we have 
\begin{align*}
\text{Cap}_{2,1,q'}(K_{j})\lesssim \rho_{j+3}^{-N-2+2q'} \text{Cap}_{2,1,q'}(K\cap \tilde{Q}_{\rho_{j-2}}(0,0)),
\end{align*} 
which implies \eqref{6h0106201411}. 
\end{proof}
%%%%%%%%%%%%%%%%%%%%%%%%%%%%%%%%%%%%%%%%%%%%%%%%%%%%%%%%%%%%%%%%%%%%%%%%%%%%%%%%%%%%%%%%%%%%%%%%%%%%%%%%%%%%%%%%%%%%%%%%%%%%%%%%%%%%%%%%%%%%%%%%%%%%%%%%%%%%%%%%%%%%%%%%%%%%%%

\begin{proposition}\label{6h010620144}
Let $K,u,\xi$ be as in Proposition \ref{6h010620143}. For any compact set $K_0$ in $\overline{\tilde{Q}_1(0,0)}$ with positive measure $|K_0|$, there exists $\varepsilon=\varepsilon(N,q,|K_0|)>0$ such that 
\begin{align*}
\operatorname{Cap}_{2,1,q'}(K)\leq \varepsilon \Rightarrow \inf_{K_0}u \lesssim \int_{\tilde{Q}_2(0,0)}u\xi dx dt,
\end{align*}
where the constant in the inequality $\lesssim$ depends on $K_0$. In particular, 
\begin{align}\label{6h010620145}
\operatorname{Cap}_{2,1,q'}(K)\leq \varepsilon \Rightarrow
\inf_{K_0}u \lesssim \operatorname{Cap}_{2,1,q'}(K) +R^{-\frac{2}{q-1}}.
\end{align}
\end{proposition}
\begin{proof}
It is enough to prove that there exists  $\varepsilon>0$ such that 
\begin{align}\label{6h010620142}
\text{Cap}_{2,1,q'}(K)\leq \varepsilon \Rightarrow |K_1|\geq 1/2 |K_0|,
\end{align}
where $K_1=\{(x,t)\in K_0:\xi(x,t)\geq 1/2\}.$
By \eqref{6h010620141} in Proposition \ref{6h280520141}, we have the following estimates
$$
 |K_0\backslash K_1|^{1-\frac{2q'}{N+2}}\lesssim \text{Cap}_{2,1,q'}(K_0\backslash K_1), $$
if $q>q_*$, and
 $$\left(\log\left(\frac{|\tilde{Q}_{200}(0,0)|}{|K_0\backslash K_1|}\right)\right)^{-\frac{N}{2}}\lesssim \text{Cap}_{2,1,q'}(K_0\backslash K_1),
$$
if $q=q_*$. On the other hand, 
\begin{align*}
\text{Cap}_{2,1,q'}(K_0\backslash K_1)&=\text{Cap}_{2,1,q'}(\{ K_0:\varphi >1-(1/2)^{1/(2q')} \})\\& \leq (1-(1/2)^{1/(2q')})^{-q'} \int_{\mathbb{R}^{N+1}}\left(|D^2 \varphi |^{q'}+|\nabla \varphi|^{q'}+|\varphi|^{q'}+|\partial_t\varphi |^{q'}\right) dxdt\\&\lesssim \text{Cap}_{2,1,q'}(K),
\end{align*}
where $\varphi$ is in Proposition \ref{6h010620143}. 
Henceforth, one can find $\varepsilon=\varepsilon(N,q,|K_0|)>0$ such that 
\begin{align*}
\text{Cap}_{2,1,q'}(K)\leq \varepsilon \Rightarrow |K_0\backslash K_1|\leq 1/2~ |K_0|.
\end{align*}
This implies \eqref{6h010620142}.
\end{proof}\\
%%%%%%%%%%%%%%%%%%%%%%%%%%%%%%%%%%%%%%%%%%%%%%%%%%%%%%%%%%%%%%%%%%%%%%%%%%%%%%The large solutions%%%%%%%%%%%%%%%%%%%%%%%%%%%%%%%%%%%%%%%%%%%%%%%%%%%%%%%%%%%%%%%%%%%%%%%%%%%%%%%%%%%%%%%%%%%%%%%%%%%%%%%%%%%%%%%%%%%%%%%%%%%%%%%%%%%%%%%%%%%%%%%%%%%%%%%%%%%%%%%%%%%%%%%%%%%%%%%%%%%%%%%%%%%%%%%%%%%%%%%%%%%%%%%%%%%%%%%%%%%%%%%%%%%%%%%%%%%%%%%%%%%%%%%%%%%%%%%%%%%%%%%%%%%%%%%%%%%%%%%%%%%%%%%%%%%%%%%%%%%%%%%%%%%%%%%%%%%%%%%%%%%%%%%%%%%%%%%%%%%%%%%%%%%%%%%%%%%%%%%%%%%%%%%%%%%%%%%%%%%%%%%%%%%%%%%%%%%%%%%%%%%%%%%%%%%%%%%%%%%%%%%%%%%%%%%%%%%%%%%%%%%%%%%%%%%%%%%%%%%%%%%%%%%%%%%%%%%%%%%%%%%%%
\section{Large solutions}
In the first part of this section,  we prove  theorem \ref{6h230520143}-(ii), then we  prove  theorems \ref{6h230520143}-(i) and \ref{6h230520144}. At end weapply our result to a parabolic viscous Hamilton-Jacobi equation. 
%%%%%%%%%%%%%%%%%%%%%%%%%%%%%%%%%%%%%%%%%%%%%%%%%%%%%%%%%%%%%%%%%%%%%%%%%%%%%%%%%%%%%%%%%%%%%%%%%%%%%%%%%%%%%%%%%%%%
\subsection{Proof of Theorem \ref{6h230520143}-(ii)}
Let $R_0\geq 4$ such that $O\subset\subset \tilde{Q}_{R_0}(0,0)$. Assume that the equation \eqref{6h220520149} has a large solution $u$. We claim that \eqref{6h0106201413} holds with $(x,t)\in \partial_pO$, and without loss of generality, we can assume $(x,t)=(0,0)$. Set $K=\tilde{Q}_{2R_0}(0,0)\backslash O$ and define 
 \begin{align*}
& T_j=\{x:\rho_{j+1}\leq\max\{|x|,|t|^{1/2}\}\leq \rho_{j}, t\leq 0\},\\&
 \tilde{T}_j=\{x:\rho_{j+3}\leq\max\{|x|,|t|^{1/2}\}\leq \rho_{j-2},t\le 0\}. 
 \end{align*}
Here $\rho_j=2^{-j}$.  For $j\geq 3$, let $u_1,u_2,u_3,u_4$ be the maximal solutions of  \eqref{6h290520141} when $K$ is replaced by  $K\cap \overline{Q_{\rho_{j+3}}(0,0)}$,   $K\cap\tilde{T}_j$, $\left(K\cap \overline{Q_{1}(0,0)}\right)\backslash  Q_{\rho_{j-2}}(0,0)$ and $K\backslash Q_{1}(0,0)$ respectively and $R\geq 100 R_0$. From \eqref{6h010620148} in Remark \ref{6h010620147}, we can assert that 
 \begin{align*}
 u\leq u_1+u_2+u_3+u_4 ~~\text{ in }~ O\cap \{(x,t)\in\mathbb{R}^{N+1}:t\leq 0\}.
 \end{align*}
Thus, 
 \begin{align}\label{6h020620149}
  \inf_{T_j}u\leq|| u_1||_{L^\infty(T_j)}+|| u_3||_{L^\infty(T_j)}+|| u_4||_{L^\infty(T_j)}+\inf_{T_j}u_2.
  \end{align}
  \textbf{Case 1}: $q>q_*$. By \eqref{6h310520145} in Remark \ref{6h010620147}, 
  \begin{align}\label{6h020620141}
  || u_4||_{L^\infty(T_j)}\lesssim 1.
  \end{align}
  By \eqref{6h0106201411} in Proposition \ref{6h0106201414}, 
  \begin{align}\label{6h020620142}
  ||u_3||_{L^\infty(T_j)}\lesssim \sum_{i=-2}^{j-4}\frac{\text{Cap}_{2,1,q'}(K\cap Q_{\rho_{i}}(0,0))}{\rho_i^N}+j R^{-\frac{2}{q-1}}.
  \end{align}
 Since  $(x,t)\mapsto \overline{u}_1(x,t)=\rho_{j+3}^{2/(q-1)}u_1(\rho_{j+3}x,\rho_{j+3}^2t)$ is the maximal solution of  \eqref{6h290520141} when $(K, \tilde{Q}_R(0,0))$ is replaced by  $(\{(y/\rho_{j+3},s/\rho_{j+3}^2):(y,s)\in K\cap \overline{Q_{\rho_{j+3}}(0,0)} \}, \tilde{Q}_{R/\rho_{j+3}}(0,0))$, we derive 
 \begin{align*}
 ||\overline{u}_1||_{L^\infty(T_{-3})}\lesssim \frac{\text{Cap}_{2,1,q'}(K\cap Q_{\rho_{j+2}}(0,0))}{\rho_j^{N+2-2q'}}+ (R/\rho_{j+3})^{-\frac{2}{q-1}},
 \end{align*}
thanks to  \eqref{6h310520144} in Proposition \ref{6h010620143}  and \eqref{6h010620149} in Proposition \ref{6h280520141}, from which follows 
\begin{align}\label{6h020620143}
 ||u_1||_{L^\infty(T_{j})}\lesssim \frac{\text{Cap}_{2,1,q'}(K\cap Q_{\rho_{j+2}}(0,0))}{\rho_j^{N}}+ R^{-\frac{2}{q-1}}.
 \end{align} 
   Since  $(x,t)\mapsto \overline{u}_2(x,t)=\rho_{j-2}^{2/(q-1)}u_2(\rho_{j-2}x,\rho_{j-2}^2t)$ is the maximal solution of  \eqref{6h290520141} when the couple $(K, \tilde{Q}_R(0,0))$ is replaced by  $(\{(y/\rho_{j-2},s/\rho_{j-2}^2):(y,s)\in K\cap\tilde{T}_j \}, \tilde{Q}_{R/\rho_{j-2}}(0,0))$, 
  Proposition \ref{6h010620144} and relation \eqref{6h010620149} in Proposition \ref{6h280520141} yield
  \begin{align*}
  \frac{\text{Cap}_{2,1,q'}(K\cap\tilde{T}_j )}{\rho_{j-2}^{N+2-2q'}}\leq \varepsilon \Rightarrow
  \inf_{T_2}\overline{u}_2 \lesssim \frac{\text{Cap}_{2,1,q'}(K\cap\tilde{T}_j )}{\rho_{j-2}^{N+2-2q'}} +(R/\rho_{j-2})^{-\frac{2}{q-1}},
  \end{align*}
  which implies 
  \begin{align}\label{6h020620144}
    \frac{\text{Cap}_{2,1,q'}(K\cap Q_{\rho_{j-3}}(0,0)  )}{\rho_{j-2}^{N+2-2q'}}\leq \varepsilon \Rightarrow
    \inf_{T_j}u_2 \lesssim \frac{\text{Cap}_{2,1,q'}(K\cap Q_{\rho_{j-3}}(0,0) )}{\rho_{j-2}^{N}}+ R^{-\frac{2}{q-1}},
    \end{align}
    for some $\varepsilon=\varepsilon(N,q)>0$.\\
    First, we assume that there exists $J\in\mathbb{N}$, $J\geq 10$ such that 
    \begin{align*}
    \frac{\text{Cap}_{2,1,q'}(K\cap Q_{\rho_{j-3}}(0,0)  )}{\rho_{j-2}^{N+2-2q'}}\leq \varepsilon \quad\forall ~j\geq J.
    \end{align*}
    Then, from \eqref{6h020620149} and \eqref{6h020620141}, \eqref{6h020620142}, \eqref{6h020620143}, \eqref{6h020620144}, we have 
    \begin{align*}
     \inf_{T_j}u \lesssim \sum_{i=-2}^{j+2}\frac{\text{Cap}_{2,1,q'}(K\cap Q_{\rho_i}(0,0))}{\rho_i^N}+j R^{-\frac{2}{q-1}}+1,
      \end{align*}
 for any $j\geq J$, and letting $R\to \infty$, 
      \begin{align*}
           \inf_{T_j}u \lesssim \sum_{i=-2}^{j+2}\frac{\text{Cap}_{2,1,q'}(K\cap Q_{\rho_i}(0,0))}{\rho_i^N}+1.
            \end{align*}
Since $\inf_{T_j}u\to \infty $ as $j\to\infty$, we get
 \begin{align*}
 \sum_{i=0}^{\infty}\frac{\text{Cap}_{2,1,q'}(K\cap Q_{\rho_i}(0,0))}{\rho_i^N}=\infty,
 \end{align*}
 which implies that \eqref{6h0106201413} holds with $(x,t)=(0,0)$.\\
 Alternatively, assume that for infinitely many $j$ 
 \begin{align*}
     \frac{\text{Cap}_{2,1,q'}(K\cap Q_{\rho_{j-3}}(0,0)  )}{\rho_{j-2}^{N+2-2q'}}> \varepsilon,
     \end{align*}
 then, 
  \begin{align*}
      \frac{\text{Cap}_{2,1,q'}(K\cap Q_{\rho_{j-3}}(0,0)  )}{\rho_{j-2}^{N}}> \rho_{j-2}^{2-2q'}\varepsilon \to \infty ~~\text{ when }~ j\to \infty.
      \end{align*}
We also derive that \eqref{6h0106201413} holds with $(x,t)=(0,0)$. This proves the case $q>q_*$. \smallskip

\noindent \textbf{Case 2:} $q=q_*$.  Similarly to Case 1, we have: for $j\geq 6$
 \begin{align}\label{6h020620145}
  & || u_4||_{L^\infty(T_j)}\lesssim 1,\\&\label{6h020620146}
  ||u_3||_{L^\infty(T_j)}\lesssim \sum_{i=0}^{j-2}\frac{\text{Cap}_{2,1,q'}(K_j)}{\rho_i^N}+j R^{-\frac{2}{q-1}}, \\& \label{6h020620147}
  ||u_1||_{L^\infty(T_{j})}\lesssim \frac{\text{Cap}_{2,1,q'}(K_{j})}{\rho_j^{N}}+ R^{-\frac{2}{q-1}},  \\& \label{6h020620148}
 \text{Cap}_{2,1,q'}(K_{j-5} )\leq \varepsilon \Rightarrow
      \inf_{T_j}u_2 \lesssim \frac{\text{Cap}_{2,1,q'}(K_{j-5} )}{\rho_{j}^{N}}+ R^{-\frac{2}{q-1}},
  \end{align}
  where $K_j=\{(x/\rho_{j+3},t/\rho_{j+3}^2): (x,t)\in K\cap Q_{\rho_{j-3}}(0,0)  \}$  and $\varepsilon=\varepsilon(N)>0$.\\
From \eqref{6h010620149} in Proposition \ref{6h280520141}, we have 
  \begin{align*}
  \frac{1}{\text{Cap}_{2,1,q'}(K\cap Q_{\rho_{j-3}}(0,0))}\leq  \frac{c}{\text{Cap}_{2,1,q'}(K_j)}+cj^{N/2}
  \end{align*}
  for any $j\geq 4$ where $c=c(N)$.
If there are infinitely many $j\geq 4$ such that 
\begin{align*}
\text{Cap}_{2,1,q'}(K\cap Q_{\rho_{j-3}}(0,0))> \frac{1}{2cj^{N/2}},
\end{align*}
then \eqref{6h0106201413} holds with $(x,t)=(0,0)$ since
\begin{align*}
\frac{\text{Cap}_{2,1,q'}(K\cap Q_{\rho_{j-3}}(0,0))}{\rho_{j-3}^N} > \frac{2^{j-3}}{2cj^{N/2}}
\to \infty ~~\text{ when } j\to\infty. 
\end{align*}
Now, we assume that there exists $J\geq 6$ such that 
\begin{align*}
\text{Cap}_{2,1,q'}(K\cap Q_{\rho_{j-3}}(0,0))\leq \frac{1}{2cj^{N/2}}~~\forall ~j\geq J.
\end{align*}
Then, 
\begin{align*}
\text{Cap}_{2,1,q'}(K_j)\leq 2c \text{Cap}_{2,1,q'}(K\cap Q_{\rho_{j-3}}(0,0))~~\forall ~j\geq J.
\end{align*}
This leads to
\begin{align*}
\text{Cap}_{2,1,q'}(K_j)\leq 2c \text{Cap}_{2,1,q'}(K\cap Q_{\rho_{j-3}}(0,0))\leq \varepsilon~~~~\forall ~j\geq J'+J,
\end{align*}
 for some $J'=J'(N)$. 
Hence, from \eqref{6h020620145}-\eqref{6h020620148} we have, for any $j\geq J'+J+3$,
  \begin{align*}
    & || u_4||_{L^\infty(T_j)}\lesssim 1,\\&
    ||u_3||_{L^\infty(T_j)}\lesssim \sum_{i=J'+J+1}^{j-2}\frac{\text{Cap}_{2,1,q'}(K\cap Q_{\rho_{i-3}}(0,0))}{\rho_i^N}+C(J'+J)+j R^{-\frac{2}{q-1}}, \\& 
    ||u_1||_{L^\infty(T_{j})}\lesssim \frac{\text{Cap}_{2,1,q'}(K\cap Q_{\rho_{j-3}}(0,0))}{\rho_j^{N}}+ R^{-\frac{2}{q-1}},  \\& ~~
        \phantom{--}\inf_{T_j}u_2 \lesssim \frac{\text{Cap}_{2,1,q'}(K\cap Q_{\rho_{j-8}}(0,0))}{\rho_{j}^{N}}+ R^{-\frac{2}{q-1}},
    \end{align*}
  where $C(J'+J)=\sum_{i=0}^{J'+J}\frac{\text{Cap}_{2,1,q'}(K_j)}{\rho_i^N}$.\\ Consequently we derive 
  \begin{align*}
  \inf_{T_j}u\lesssim \sum_{i=0}^{j}\frac{\text{Cap}_{2,1,q'}(K\cap Q_{\rho_{i}}(0,0))}{\rho_i^N}+C(J'+J)+1+j R^{-\frac{2}{q-1}}~~\forall ~j\geq J'+J+3
  \end{align*}
from \eqref{6h020620149}.  Letting $R\to\infty$ and $j\to \infty$ we obtain 
   \begin{align*}
     \sum_{i=0}^{\infty}\frac{\text{Cap}_{2,1,q'}(K\cap Q_{\rho_{i}}(0,0))}{\rho_i^N}=\infty,
     \end{align*} 
     i.e. \eqref{6h0106201413} holds with $(x,t)=(0,0)$.  This completes the proof of Theorem \ref{6h230520143}-(ii).
     %\\
     %%%%%%%%%%%%%%%%%%%%%%%%%%%%%%%%%%%%%%%%%%%%%%%%%%%%%%%%%%%%%%%%%%%%%%%%%%%%%%%%%%%%%%%%%%%%%%%%%%%%%%%%%%%%%%%%%%%%%%%%%%%%%%%%%%%%%%%%%%%%%%%%%%%%%%%%%%%%%%%%%%%%%%%%%%%%
\subsection{Proof of Theorem \ref{6h230520143}-(i) and Theorem \ref{6h230520144}}
Fix $(x_0,t_0)\in \partial_p O$. We can assume that $(x_0,t_0)=0$. Let $\delta\in (0,1/100)$. For $(y_0,s_0)\in (B_\delta(0)\times(-\delta^2,\delta^2))\cap O$, we set
$$
M_k=O^c\cap \left(\overline{B_{r_{k+2}}(y_0)}\times [s_0-(73+\frac{1}{2})r_{k+2}^2,s_0-(70+\frac{1}{2})r_{k+2}^2]\right),$$
and
$$
S_k=\{(x,t):r_{k+1}\leq \max\{|x-y_0|,|t-s_0|^{\frac{1}{2}}\}<r_k\}\text{ for }~k=1,2,...,
$$
where $r_k=4^{-k}$. Note that $M_k=\emptyset$ for $k$ large enough and $M_k\subset S_k$ for all $k$. Let $R_0\geq 4$ such that $O\subset\subset \tilde{Q}_{R_0}(0,0)$.  By Theorems \ref{6h2305201411} and \ref{6h2305201412} and estimate \eqref{6h230520142} there exist two sequences $\{\mu_k\}_k$ and $\{\nu_k\}_k$ of nonnegative Radon measures such that 
\begin{equation}\label{6h210520145}
~~~~~~~~~~\text{supp} (\mu_k)\subset M_k, ~\text{supp} (\nu_k)\subset M_k,
\end{equation}
\begin{equation}\label{6h210520146}
\mu_k(M_k)\asymp \text{Cap}_{2,1,q'}(M_k)\asymp \int_{\mathbb{R}^{N+1}}\left(\mathbb{I}_2^{2R_0}[\mu_k]\right)^qdxdt
\end{equation}
and
\begin{equation}\label{6h210520147}
\nu_k(M_k)\asymp \mathcal{PH}_1^N(M_k), ~~||\mathbb{M}_1^{2R_0}[\nu_k]||_{L^\infty(\mathbb{R}^{N+1})}\leq 1~~\text{ for } k=1,2,...,
\end{equation}
where the constants of equivalence depend on $N,q, R_0$. \smallskip

\noindent Take $\varepsilon>0$ such that $ \exp\left(C_1 \varepsilon\mathbb{I}_2^{2R_0}[\sum_{k=1}^{\infty}\nu_k]\right)\in L^1(\tilde{Q}_{R_0}(0,0))$, in which expression the constant $C_1=C_1(N)$ is the one of inequality  \eqref{6h230520145}. By Theorem \ref{6h230520148} and Proposition \ref{6h230520149}, there exist two nonnegative solutions $U_1,U_2$ of problems 
$$
 \begin{array}{lll}
     \partial_tU_1-\Delta U_1 +U_1^q=\displaystyle\varepsilon\sum_{k=1}^{\infty}\mu_k \qquad&\text{ in } \tilde{Q}_{R_0}(0,0),\\ 
     \phantom{ \partial_tU_1-\Delta  +U_1^q}
      U_1=0~&\text{ on }~\partial_p \tilde{Q}_{R_0}(0,0),   
      \end{array} 
$$
and 
$$
\begin{array}{lll}
         \partial_tU_2-\Delta U_2 +e^{U_2}-1=\displaystyle \varepsilon\sum_{k=1}^{\infty}\nu_k\qquad&\text{ in } \tilde{Q}_{R_0}(0,0),\\ 
              \phantom{\partial_tU_2-\Delta  +e^{U_2}-1}
          U_2=0&\text{ on }~\partial_p \tilde{Q}_{R_0}(0,0),      
          \end{array} 
$$
        respectively which satisfy 
        \begin{align}  \nonumber      
        & U_1(y_0,z_0)\gtrsim   \sum_{i=0}^{\infty}\sum_{k=1}^{\infty}\varepsilon\frac{\mu_k(B_{\frac{r_i}{8}}(y_0)\times(s_0-\frac{37}{128}r_i^2,s_0-\frac{35}{128}r_i^2))}{r_i^N}\\&
        ~~~~~~~~~~~~~~~~- \mathbb{I}_2^{2R_0}\left[\left(\mathbb{I}_2^{2R_0}[\varepsilon
        \sum_{k=1}^{\infty}\mu_k]\right)^q\right](y_0,s_0)=:A,\label{6h210520141}
        \end{align}
        and 
\begin{align}  \nonumber      
        & U_2(y_0,z_0)\gtrsim  \sum_{i=0}^{\infty}\sum_{k=1}^{\infty}\varepsilon\frac{\nu_k(B_{\frac{r_i}{8}}(y_0)\times(s_0-\frac{37}{128}r_i^2,s_0-\frac{35}{128}r_i^2))}{r_i^N}\\&
        ~~~~~~~~~~~~~~~~- \mathbb{I}_2^{2R_0}\left[\exp\left(C_1\mathbb{I}_2^{2R_0}[\varepsilon
        \sum_{k=1}^{\infty}\nu_k]\right)-1\right](y_0,s_0)=:B,\label{6h210520142}
        \end{align} 
         and $U_1,U_2\in C^{2,1}(O)$.\\
    Let $u_1,u_2$ be the maximal solutions of equations \eqref{6h200520143} and \eqref{6h200520145} respectively.\\ We have $u_1(y_0,s_0)\geq U_1(y_0,s_0)$ and  $u_2(y_0,s_0)\geq U_2(y_0,s_0)$. 
    Now, we claim that 
    \begin{equation}
\label{6h210520143} A \gtrsim \sum_{k=1}^{\infty}\frac{\text{Cap}_{2,1,q'}(M_k)}{r_k^N},
\end{equation}
and
       \begin{equation}\label{6h210520144}
   B\gtrsim  -c_1(R_0)+\sum_{k=1}^{\infty}\frac{\mathcal{PH}^N_1(M_k)}{r_k^N}.
   \end{equation}
\noindent
    \textbf{Proof of assertion \eqref{6h210520143}.}  From \eqref{6h210520146} we have 
    \begin{align}\label{6h220520143}
    A\gtrsim  \varepsilon \sum_{k=1}^{\infty}\frac{\text{Cap}_{2,1q'}(M_k)}{r_k^N}-\varepsilon^q A_0,
    \end{align} 
    with \begin{align*}
    A_0=\mathbb{I}_2^{2R_0}\left[\left(\mathbb{I}_2^{2R_0}[
                \sum_{k=1}^{\infty}\mu_k]\right)^q\right](y_0,s_0).
    \end{align*} 
 Take $i_0\in\mathbb{Z}$ such that $r_{i_0+1}<\max\{2R_0,1\}\leq r_{i_0}$. Then 
 \begin{align*}
 A_0&\lesssim\sum_{i=i_0}^{\infty}r_i^{-N}\int_{\tilde{Q}_{r_i}(y_0,s_0)}\left(\mathbb{I}_2^{2R_0}[
                 \sum_{k=1}^{\infty}\mu_k]\right)^qdxdt\\[2mm]&
    = \sum_{i=i_0}^{\infty}\sum_{j=i}^{\infty}r_i^{-N}\int_{S_j}\left(\mathbb{I}_2^{2R_0}[
                     \sum_{k=1}^{\infty}\mu_k]\right)^qdxdt\\[2mm]& 
= \sum_{j=k_0}^{\infty}\sum_{i=i_0}^{j}r_i^{-N}\int_{S_j}\left(\mathbb{I}_2^{2R_0}[
                     \sum_{k=1}^{\infty}\mu_k]\right)^qdxdt 
\\[2mm]
&\lesssim   \sum_{j=i_0}^{\infty}r_j^{-N}\int_{S_j}\left(\mathbb{I}_2^{2R_0}[
                     \sum_{k=1}^{\infty}\mu_k]\right)^qdxdt.                                     
 \end{align*}
 Here we have used the fact that $\sum_{i=i_0}^{j}r_i^{-N}\leq 
 \frac{4}{3}r_j^{-N}$ for all $j$.\\
If we set $\mu_k\equiv0$ for all $i_0-1\leq k\leq 0$,
the previous inequality becomes
\begin{align}\nonumber
 A_0&\lesssim   \sum_{j=i_0}^{\infty}r_j^{-N}\int_{S_j}\left(\mathbb{I}_2^{2R_0}[
                     \mu_j+\sum_{k=i_0-1}^{j-1}\mu_k+ \sum_{k=j+1}^{\infty}\mu_k]\right)^qdxdt \\&\nonumber \lesssim \sum_{j=i_0}^{\infty} r_j^{-N} \int_{S_j}\left(\mathbb{I}_2^{2R_0}[
                     \mu_j]\right)^qdxdt\\&\nonumber +\sum_{j=i_0}^{\infty}r_j^{2} \left(\sum_{k=i_0-1}^{j-1}||\mathbb{I}_2^{2R_0}[
                                          \mu_k]||_{L^\infty(S_j)}\right)^q
\\&\nonumber+\sum_{j=i_0}^{\infty}r_j^{2} \left(\sum_{k=j+1}^{\infty}||\mathbb{I}_2^{2R_0}[
                                          \mu_k]||_{L^\infty(S_j)}\right)^q  
       \\&=A_1+A_2+A_3.   \label{6h220520142}                                     
\end{align} 
Using \eqref{6h210520146} we obtain 
\begin{align}\label{6h220520141}
A_1\leq \sum_{k=1}^{\infty}\frac{\text{Cap}_{2,1,q'}(M_k)}{r_k^N}. 
\end{align} 
Next, using \eqref{6h210520145} we have for any $(x,t)\in S_j$ 

\begin{align}\label{6h220520145}
\mathbb{I}_2^{2R_0}[\mu_k](x,t)=\int_{r_{j+1}}^{2R_0}\frac{\mu_k(\tilde{Q}_\rho(x,t))}{\rho^N}\frac{d\rho}{\rho}\lesssim\frac{\mu_k(\mathbb{R}^{N+1})}{r_j^N}
\end{align} 
if $k\geq j+1$, and 
\begin{align}\label{6h220520146}
\mathbb{I}_2^{2R_0}[\mu_k](x,t)=\int_{r_{k+1}}^{2R_0}\frac{\mu_k(\tilde{Q}_\rho(x,t))}{\rho^N}\frac{d\rho}{\rho}\lesssim\frac{\mu_k(\mathbb{R}^{N+1})}{r_k^N}
\end{align}
if $k\leq j-1$. Thus,
$$
A_2\lesssim\sum_{j=i_0}^{\infty}r_j^{2} \left(\sum_{k=i_0-1}^{j-1}\frac{\mu_k(\mathbb{R}^{N+1})}{r_k^N}\right)^q,$$
and
$$
A_3\lesssim\sum_{j=i_0}^{\infty}r_j^{2-Nq} \left(\sum_{k=j+1}^{\infty}\mu_k(\mathbb{R}^{N+1})\right)^q.
$$
Noticing that $(a+b)^q-a^q\leq q(a+b)^{q-1}b$ for any $a,b\geq 0$, we get
\begin{align*}
& (1-4^{-2})\sum_{j=i_0}^{\infty}r_j^{2} \left(\sum_{k=i_0-1}^{j-1}\frac{\mu_k(\mathbb{R}^{N+1})}{r_k^N}\right)^q
\\&\phantom{--------}
= \sum_{j=i_0}^{\infty}r_j^{2} \left(\sum_{k=i_0-1}^{j-1}\frac{\mu_k(\mathbb{R}^{N+1})}{r_k^N}\right)^q- \sum_{j=i_0+1}^{\infty}r_j^{2} \left(\sum_{k=i_0-1}^{j-2}\frac{\mu_k(\mathbb{R}^{N+1})}{r_k^N}\right)^q
\\&\phantom{--------}
\leq\sum_{j=i_0}^{\infty}qr_j^{2} \left(\sum_{k=i_0-1}^{j-1}\frac{\mu_k(\mathbb{R}^{N+1})}{r_k^N}\right)^{q-1}\frac{\mu_{j-1}(\mathbb{R}^{N+1})}{r_{j-1}^N}.
\end{align*} 
Similarly, we also have 
 \begin{align*}
 &(1-4^{2-Nq})\sum_{j=i_0}^{\infty}r_j^{2-Nq} \left(\sum_{k=j+1}^{\infty}\mu_k(\mathbb{R}^{N+1})\right)^q\\&~~~~~~~~~~~~~~~~\leq \sum_{j=i_0}^{\infty}qr_j^{2-Nq}\left(\sum_{k=j+1}^{\infty}\mu_k(\mathbb{R}^{N+1})\right)^{q-1}\mu_{j+1}(\mathbb{R}^{N+1}).
 \end{align*}
 Therefore,
 \begin{align*}
 A_2+A_3&\lesssim \sum_{j=i_0}^{\infty}r_j^{2} \left(\sum_{k=i_0-1}^{j-1}\frac{\mu_k(\mathbb{R}^{N+1})}{r_k^N}\right)^{q-1}\frac{\mu_{j-1}(\mathbb{R}^{N+1})}{r_{j-1}^N}\\&~~~+\sum_{j=i_0}^{\infty}r_j^{2-Nq}\left(\sum_{k=j+1}^{\infty}\mu_k(\mathbb{R}^{N+1})\right)^{q-1}\mu_{j+1}(\mathbb{R}^{N+1}).
 \end{align*}
Since $\mu_k(\mathbb{R}^{N+1})\lesssim r_k^{N+2-2q'}$ if $q>q_*$ and $\mu_k(\mathbb{R}^{N+1})\lesssim \min\{k^{-\frac{1}{q-1}},1\}$ if $q=q_*$  for any $k$, we infer that
$$
 r_j^{2} \left(\sum_{k=i_0-1}^{j-1}\frac{\mu_k(\mathbb{R}^{N+1})}{r_k^N}\right)^{q-1}\lesssim 1,
 $$
 and
 $$
 r_j^{2-Nq}\left(\sum_{k=j+1}^{\infty}\mu_k(\mathbb{R}^{N+1})\right)^{q-1}\lesssim r_{j+1}^{-N}~~\text{ for any }~j.
$$
In the  case $q=q_*$ we  assume $N\geq 3$  in order to ensure that  
$$\sum_{j=1}^{\infty}\mu_k(\mathbb{R}^{N+1})\lesssim\sum_{k=1}^{\infty}k^{-\frac{1}{q-1}}<\infty.$$
 This leads to 
 \begin{align*}
 A_2+A_3\lesssim \sum_{k=1}^{\infty}\frac{\mu_k(\mathbb{R}^{N+1})}{r_k^N}.
 \end{align*}
Combining this with  \eqref{6h220520141} and \eqref{6h220520142}, we deduce
\begin{align*}
A_0\lesssim \sum_{k=1}^{\infty}\frac{\text{Cap}_{2,1,q'}(M_k)}{r_k^N}.
\end{align*}
Consequently, we obtain \eqref{6h210520143} from \eqref{6h220520143}, for $\varepsilon$ small enough. \smallskip
%%%%%%%%%%%%%%%%%%%%%%%%%%%%%%%%%%%%%%%%%%%%%%%%%%%%%%%%%%%%%%%%%%%%%%%%%%%%%%%%%%%%%%%%%%%%%%%%%%%%%%%%%%%%%%%%%%%%%%%%%%%%%%%%%%%%%%%%%%%%%%%%%%%%%%%%%%%%%%%%%%%%%%%%%%%%%%

\noindent\textbf{Proof of assertion \eqref{6h210520144}.} 
From \eqref{6h210520147} we get 
\begin{align*}
B\gtrsim \varepsilon\sum_{k=1}^{\infty}\frac{\mathcal{PH}_1^N(M_k)}{r_k^N}-B_0,
\end{align*}
where 
\begin{align*}
B_0=\mathbb{I}_2^{2R_0}\left[\exp\left(C_1\mathbb{I}_2^{2R_0}[\varepsilon
        \sum_{k=1}^{\infty}\nu_k]\right)-1\right](y_0,s_0).
\end{align*}
We show that 
\begin{align}\label{6h220520147}
B_0\leq c(N,q,R_0)~~\text{ for } \varepsilon ~\text{ small enough.}
\end{align}
In fact, as above we have 
\begin{align*}
B_0\lesssim\sum_{j=i_0}^{\infty}r_j^{-N}\int_{S_j}\exp\left(C_1\varepsilon\mathbb{I}_2^{2R_0}[
                     \sum_{k=1}^{\infty}\nu_k]\right)dxdt.
\end{align*} 
Consequently,
\begin{align}\nonumber
 B_0& \lesssim\sum_{j=i_0}^{\infty} r_j^{-N} \int_{S_j}\exp\left(3C_1\varepsilon\mathbb{I}_2^{2R_0}[
                     \nu_j]\right)dxdt\\&\nonumber +\sum_{j=i_0}^{\infty}r_j^{2} \exp\left(3C_1\varepsilon\sum_{k=i_0-1}^{j-1}||\mathbb{I}_2^{2R_0}[
                                          \nu_k]||_{L^\infty(S_j)}\right)
\\&\nonumber+\sum_{j=i_0}^{\infty}r_j^{2} \exp\left(3C_1\varepsilon\sum_{k=j+1}^{\infty}||\mathbb{I}_2^{2R_0}[\nu_k]||_{L^\infty(S_j)}\right) 
       \\&=B_1+B_2+B_3.\label{6h220520144}                              
\end{align}
Here we have used the convexity inequality $3\exp(a+b+c)\leq \exp(3a)+\exp(3b)+\exp(3c)$ for all real numbers $a,b,c$.\\
By Theorem \ref{6h2305201410}, we have
\begin{align*}
\int_{S_j}\exp\left(3C_1\varepsilon\mathbb{I}_2^{2R_0}[
                     \nu_j]\right)dxdt\lesssim r_j^{N+2}~~\text{for all} ~j,
\end{align*}
for $\varepsilon>0$ small enough. 
Hence, 
\begin{align}\label{6h220520148}
B_1\lesssim\sum_{j=i_0}^{\infty} r_j^{2}\lesssim(\max\{2R_0,1\})^2.
\end{align}
Note that estimates \eqref{6h220520145}  and \eqref{6h220520146} are also true with $\nu_k$; we deduce 
\begin{align*}
B_2+B_3&\lesssim \sum_{j=i_0}^{\infty}r_j^{2} \exp\left(c_2\varepsilon\sum_{k=i_0-1}^{j-1}\frac{\mu_k(\mathbb{R}^{N+1})}{r_k^N}\right)\\&~~
+\sum_{j=i_0}^{\infty}r_j^{2} \exp\left(c_2\varepsilon\sum_{k=j+1}^{\infty}\frac{\mu_k(\mathbb{R}^{N+1})}{r_j^N}\right).                                          
\end{align*}
From \eqref{6h210520147}  we have $\mu_k(\mathbb{R}^{N+1})\lesssim r_k^N$ for all $k$, therefore
\begin{align*}
B_2+B_3&\lesssim \sum_{j=i_0}^{\infty}r_j^{2} \exp\left(c_3\varepsilon (j-i_0)\right)
+\sum_{j=i_0}^{\infty}r_j^{2} \exp\left(c_{3}\varepsilon\right)
\\&\lesssim \sum_{j=i_0}^{\infty} \exp\left(c_3\varepsilon (j-i_0)-4\log(2)j\right)+r_{i_0}^2
\\&\leq c_{4}(N,q,R_0)~~\text{ for } \varepsilon~~ \text{small enough.}
\end{align*}
Combining this with \eqref{6h220520148} and \eqref{6h220520144} we obtain \eqref{6h220520147}.\\
This implies straightforwardly  $ \exp\left(C_1 \varepsilon\mathbb{I}_2^{2R_0}[\sum_{k=1}^{\infty}\nu_k]\right)\in L^1(\tilde{Q}_{R_0}(0,0))$. \\
We conclude that for any $(y_0,s_0)\in (B_\delta(0)\times(-\delta^2,\delta^2))\cap O$, 
$$
 u_1(y_0,s_0)\gtrsim \sum_{k=1}^{\infty}\frac{\text{Cap}_{2,1,q'}\left(M_k(y_0,s_0)\right)}{r_k^N},
   $$
   and
   $$
   u_2(y_0,s_0)\gtrsim -c_1(R_0)+\sum_{k=1}^{\infty}\frac{\mathcal{PH}_1^N\left(M_k(y_0,s_0)\right)}{r_k^N},
$$
 where $r_k=4^{-k}$ and
\begin{align*}
M_k(y_0,s_0)=O^c\cap \left(\overline{B_{r_{k+2}}(y_0)}\times [s_0-(73+\frac{1}{2})r_{k+2}^2,s_0-(70+\frac{1}{2})r_{k+2}^2]\right).
\end{align*}
If we take $r_{k_\delta+4}\leq \delta<r_{k_\delta+3}$, we have 
for $1\leq k\leq k_\delta$
\begin{align*}
M_k(y_0,s_0)&\supset O^c\cap \left(B_{r_{k+2}-\delta}(0)\times \left(\delta^2-(73+\frac{1}{2})r_{k+2}^2,-\delta^2-(70+\frac{1}{2})r_{k+2}^2\right)\right)\\&
\supset O^c\cap \left(B_{r_{k+3}}(0)\times \left(-73r_{k+2}^2,-71r_{k+2}^2\right)\right)
\\&
=O^c\cap \left(B_{r_{k+3}}(0)\times \left(-1168r_{k+3}^2,-1136r_{k+3}^2\right)\right).
\end{align*}
Finally,   
\begin{align*}
   & \inf_{(y_0,s_0)\in (B_\delta(0)\times(-\delta^2,\delta^2))\cap O}u_1(y_0,s_0)\\&~~~ \gtrsim -1+ \int_{r_{k_\delta+3}}^{1}\frac{\text{Cap}_{2,1,q'}(O^c\cap (B_{\rho}(0)\times (-17 b \rho^2,-b\rho^2)))}{\rho^N}\frac{d\rho}{\rho}~\text{with}~b=1136
   \\&~~~ \gtrsim -1+ \int_{30r_{k_\delta+3}}^{1}\frac{\text{Cap}_{2,1,q'}(O^c\cap (B_{\frac{\rho}{30}}(0)\times (-30 \rho^2,-\rho^2)))}{\rho^N}\frac{d\rho}{\rho}\to \infty ~\text{ as }~ \delta\to 0,
    \end{align*} 
 and
  \begin{align*}
    & \inf_{(y_0,s_0)\in (B_\delta(0)\times(-\delta^2,\delta^2))\cap O}u_2(y_0,s_0) \\&~~~ \gtrsim -1+ \int_{30r_{k_\delta+3}}^{1}\frac{\mathcal{PH}_1^N(O^c\cap (B_{\frac{\rho}{30}}(0)\times (-30 \rho^2,-\rho^2)))}{\rho^N}\frac{d\rho}{\rho}\to \infty ~\text{ as }~ \delta\to 0.
  \end{align*}
  This completes the proof of Theorem \ref{6h230520143}-(i) and Theorem \ref{6h230520144}.\\
  
  \begin{remark}[Uniqueness] In \cite{66MV0}, Marcus and V\'eron prove that condition (\ref{labut}) is not only a necessary and sufficient condition for the existence of a large solution to (\ref{u^q}), but it implies the uniqueness of a such a large solution when it is fulfilled. 
  The main step for this proof is to show that there exists a constant $c=c(\Omega,q>0)$ such that any couple of large solutions $(u,\hat u)$ satisfies
  \begin{equation}\label{equiv}
  u(x)\leq c\hat u(x)\qquad\forall x\in\Omega.
\end{equation}
The above estimate which is the key stone for proving uniqueness cannot be obtained in the case of the parabolic equation (\ref{6h220520149}) since the necessary condition and the sufficient condition in Theorem \ref{6h230520143} do not complement completely. 
  \end{remark}
  %%%%%%%%%%%%%%%%%%%%%%%%%%%%%%%%%%%%%%%%%%%%%%%%%%%%%%%%%%%%%%%%%%%%%%%%%%%%%%%%%%%%%%%%%%%%%%%%%%%%%%%%%%%%%%%%%%%%%%%%%%%%%%%%%%%%%%%%%%%%%%%%%%%%%%%%%%%%%%%%%%%%%%%%%%%%%
 \subsection{The viscous Hamilton-Jacobi parabolic equations}
In this section we apply our previous result to the question of existence of a large solution of the following type of parabolic viscous Hamilton-Jacobi equation
\begin{equation}\label{6h260520144}
\begin{array}{lll}
     \partial_tu-\Delta u +a|\nabla u|^p+bu^{q}=0 \qquad&\text{ in } O,\\ 
     \phantom{ \partial_t-\Delta u +a|\nabla u|^p+bu^{q}}
      u=\infty~&\text{ on}~\partial_p O,\\      
      \end{array} 
\end{equation}
  where $a>0,b> 0$ and $1<p\leq 2$, $q\geq 1$. 
First, we show that such a large solution to \eqref{6h260520144} does not exist when $q=1$. Equivalently, there is no function $u\in C^{2,1}(O)$ satisfying 
\begin{equation}\label{6h260520145}
\begin{array}{lll}
     \partial_tu-\Delta u +a|\nabla u|^p\geq -bu \qquad&\text{ in } O,\\ 
     \phantom{  \partial_t-\Delta u +a|\nabla u|^p}
      u=\infty~&\text{ on}~\partial_p O.\\      
      \end{array} 
\end{equation}
for $a>0$, $b> 0$ and $p>1$. Indeed, assuming that such a function $u\in C^{2,1}(O)$ exists,  we define 
$$U(x,t)=u(x,t)e^{bt}-\frac{\varepsilon}{2}|x|^2,$$ 
for $\varepsilon>0$ and 
denote by $(x_0,t_0)\in O\backslash \partial_pO$ the point where $U$ achieves it minimum in $O$, i.e.  $U(x_0,t_0)=\inf\{U(x,t):(x,t)\in O\}$.
Clearly, we have 
\begin{align*}
\partial_tU(x_0,t_0)\leq 0,~~ \Delta U(x_0,t_0)\geq 0~~\text{ and }~\nabla U(x_0,t_0)=0.
\end{align*}
Thus, 
\begin{align*}
\partial_tu(x_0,t_0)\leq -bu(x_0,t_0),~~-\Delta u(x_0,t_0)\leq -\varepsilon N e^{-bt_0} ~\text{ and }~ a|\nabla u(x_0,t_0)|^p=a\varepsilon^p|x_0|^pe^{-pbt_0},
\end{align*}
from which follows
\begin{align*}
\partial_tu(x_0,t_0)-\Delta u(x_0,t_0)+a|\nabla u(x_0,t_0)|^p&\leq -bu(x_0,t_0)+\varepsilon e^{-bt_0}\left(-N+a\varepsilon^{p-1}|x_0|^pe^{-(p-1)bt_0}\right)\\& <-bu(x_0,t_0)
\end{align*}
 for $\varepsilon$ small enough, which is a contradiction.\smallskip
 
\noindent \begin{proof}[Proof of Theorem \ref{6h260520148}] By Remark \ref{6h260520147}, we have 
\begin{align*}
\inf\{v(x,t); (x,t)\in O\}\geq (q_1-1)^{-\frac{1}{q_1-1}}R^{-\frac{2}{q_1-1}}.
\end{align*}
  Take $V=\lambda v^{\frac{1}{\alpha}}\in C^{2,1}(O)$ for $\lambda>0$. Thus $v=\lambda^{-\alpha}V^\alpha$,\begin{align*}
  \inf\{V(x,t); (x,t)\in O\}>0\}\geq\lambda (q_1-1)^{-\frac{1}{\alpha(q_1-1)}} R^{-\frac{2}{\alpha(q_1-1)}},
  \end{align*} and 
\begin{align*}
\partial_tv-\Delta v +v^{q_1}=\alpha\lambda^{-\alpha}V^{\alpha-1}\partial_tV-\alpha\lambda^{-\alpha}V^{\alpha-1}\Delta V+\alpha(1-\alpha)\lambda^{-\alpha}V^{\alpha-1}\frac{|\nabla V|^2}{V}+\lambda^{-\alpha q_1}V^{\alpha q_1}.
\end{align*}
This leads to 
\begin{align*}
\partial_tV-\Delta V+(1-\alpha)\frac{|\nabla V|^2}{V}+\alpha^{-1}\lambda^{-\alpha (q_1-1)}V^{\alpha q_1-\alpha+1}=0~~\text{ in }~O.
\end{align*}
Using H\"older's inequality we obtain
\begin{align*}
(1-\alpha)\frac{|\nabla V|^2}{V}+(2\alpha)^{-1}\lambda^{-\alpha (q_1-1)}V^{\alpha q_1-\alpha+1}&\geq c_1|\nabla V|^p\lambda^{-\frac{\alpha(q_1-1)(2-p)}{2}}V^{\frac{\alpha(q_1-1)(2-p)}{2}-(p-1)}
\\& \geq c_2|\nabla V|^p\lambda^{-(p-1)}R^{-2+p+\frac{2(p-1)}{\alpha(q_1-1)}},
\end{align*}
 and 
 \begin{align*}
 (2\alpha)^{-1}\lambda^{-\alpha (q_1-1)}V^{\alpha q_1-\alpha+1}\geq c_3\lambda^{-(q-1)}R^{-2+\frac{2(q-1)}{\alpha(q_1-1)}}V^{q}.
 \end{align*}
If we  choose $$\lambda=\min\{c_2^{\frac{1}{p-1}},c_3^{\frac{1}{q-1}}\}\min\left\{a^{-\frac{1}{p-1}}R^{-\frac{2-p}{p-1}+\frac{2}{\alpha(q_1-1)}}, b^{-\frac{1}{q-1}}R^{-\frac{2}{q-1}+\frac{2}{\alpha(q_1-1)}}\right\},$$ then 
\begin{align*}
& c_2\lambda^{-(p-1)}R^{-2+p+\frac{2(p-1)}{\alpha(q_1-1)}}\geq a,\\& 
c_3\lambda^{-(q-1)}R^{-2+\frac{2(q-1)}{\alpha(q_1-1)}}\geq b,
\end{align*}
from what follows 
\begin{align*}
\partial_tV-\Delta V+a|\nabla V|^p+bV^{q}\leq 0~~\text{ in }~O.
\end{align*}
By Remark \ref{6h260520146}, there exists a maximal solution $u\in C^{2,1}(O)$ of 
\begin{align*}
\partial_tu-\Delta u+a|\nabla u|^p+bu^{q}= 0~~\text{ in }~O.
\end{align*}
Therefore, $u\geq V=\lambda v^{\frac{1}{\alpha}}$ and $u$ is a large solution of \eqref{6h260520144}. This completes the proof of Theorem \ref{6h260520148}. 
\end{proof}\\
%%%%%%%%%%%%%%%%%%%%%%%%%%%%%%%%%%%%%%%%%%%%%%%%%%%%%%%%%%%%%%%%%%%%%%%%%%%%%%%%%%%%%%%%%%%%%%%%%%%%%%%%%%%%%%%%%%%%%%%%%%%%%%%%%%%%%%%%%%%%%%%%%%%%%%%%%%%%%%%%%%%%%%%%%%%%%%
\section{Appendix}
\begin{proof}[Proof of Proposition \ref{6h280520141}] \smallskip

\noindent{\it Step 1}. We claim that the following relation holds:
\begin{align}\label{6h0206201411}
\int_{\mathbb{R}^{N+1}}\left(\mathbb{I}_2^1[\mu](x,t)\right)^{(N+2)/N}dxdt\asymp \int_{\mathbb{R}^{N+1}}\int_{0}^{1}(\mu(\tilde{Q}_r(x,t)))^{2/N}\frac{dr}{r}d\mu(x,t).
\end{align}
In fact, we have  for $\rho_j=2^{-j}$, $j\in\mathbb{Z}$,
\begin{align*}
\sum_{j=1}^{\infty}\int_{\mathbb{R}^{N+1}}(\mu(\tilde{Q}_{\rho_j}(x,t)))^{2/N}d\mu(x,t)&\lesssim\int_{\mathbb{R}^{N+1}}\int_{0}^{1}(\mu(\tilde{Q}_r(x,t)))^{2/N}\frac{dr}{r}d\mu(x,t)\\& \lesssim \sum_{j=0}^{\infty}\int_{\mathbb{R}^{N+1}}(\mu(\tilde{Q}_{\rho_j}(x,t)))^{2/N}d\mu(x,t).
\end{align*}
Note that for any $j\in\mathbb{Z}$
\begin{align*}
\rho_j^{-N-2}\int_{\mathbb{R}^{N+1}} (\mu(\tilde{Q}_{\rho_{j+1}}(x,t)))^{(N+2)/N}dxdt&\lesssim\int_{\mathbb{R}^{N+1}}  (\mu(\tilde{Q}_{\rho_{j}}(x,t)))^{2/N}d\mu(x,t)\\&\lesssim\rho_j^{-N-2}\int_{\mathbb{R}^{N+1}} (\mu(\tilde{Q}_{\rho_{j-1}}(x,t)))^{(N+2)/N}dxdt.
\end{align*}
Thus, 
\begin{align*}
\sum_{j=2}^{\infty}\rho_j^{-N}\int_{\mathbb{R}^{N+1}}(\mu(\tilde{Q}_{\rho_j}(x,t)))^{(N+2)/N}dxdt&\lesssim\int_{\mathbb{R}^{N+1}}\int_{0}^{1}(\mu(\tilde{Q}_r(x,t)))^{2/N}\frac{dr}{r}d\mu(x,t)\\& \lesssim \sum_{j=-1}^{\infty}\rho_j^{-N}\int_{\mathbb{R}^{N+1}}(\mu(\tilde{Q}_{\rho_j}(x,t)))^{(N+2)/N}dxdt.
\end{align*} 
This yields 
\begin{align*}
\int_{\mathbb{R}^{N+1}}\left(\mathbb{M}_2^{1/4}[\mu](x,t)\right)^{(N+2)/N}dxdt &\lesssim\int_{\mathbb{R}^{N+1}}\int_{0}^{1}(\mu(\tilde{Q}_r(x,t)))^{2/N}\frac{dr}{r}d\mu(x,t)\\& \lesssim \int_{\mathbb{R}^{N+1}}\left(\mathbb{I}_2^{4}[\mu](x,t)\right)^{(N+2)/N}dxdt.
\end{align*}
By \cite[Theorem 4.2]{66H1},  \begin{align*}
\int_{\mathbb{R}^{N+1}}\left(\mathbb{M}_2^{1/4}[\mu](x,t)\right)^{(N+2)/N}dxdt\asymp \int_{\mathbb{R}^{N+1}}\left(\mathbb{I}_2^{4}[\mu](x,t)\right)^{(N+2)/N}dxdt,
\end{align*} thus we obtain \eqref{6h0206201411}.
\smallskip

\noindent{\it Step 2}. End of the proof.  
The first inequality in \eqref{6h010620141} is proved in \cite{66H1}. We now prove the second inequality.   By Theorem \ref{6h2305201412} there is $\mu\in\mathfrak{M}^+(\mathbb{R}^{N+1}), \text{supp}(\mu)\subset K$ such that 
\begin{align}
||\mathbb{M}_2^2[\mu]||_{L^\infty(\mathbb{R}^{N+1})}\leq 1~\text{ and }~\mu(K)\asymp \mathcal{PH}_2^N(K)\gtrsim |K|^{N/(N+2)}.
\end{align}
Thanks to  \eqref{6h0206201411}, we have for $\delta=\min\{1,(\mu(K))^{1/N}\}$
\begin{align*}
||\mathbb{I}_2^1[\mu]||_{L^{(N+2)/N}(\mathbb{R}^{N+1})}^{(N+2)/N}&\asymp\int_{\mathbb{R}^{N+1}}\int_{0}^{1}(\mu(\tilde{Q}_r(x,t)))^{2/N}\frac{dr}{r}d\mu(x,t)
\\& \asymp \int_{\mathbb{R}^{N+1}}\left(\int_{0}^{\delta}+\int_{\delta}^{1}\right)(\mu(\tilde{Q}_r(x,t)))^{2/N}\frac{dr}{r}d\mu(x,t)\\&\lesssim
\int_{0}^{\delta}r^2\frac{dr}{r}\int_{\mathbb{R}^{N+1}}d\mu(x,t)+\int_{\delta}^{1}\frac{dr}{r} \left(\int_{\mathbb{R}^{N+1}}d\mu(x,t)\right)^{(N+2)/N}
\\&\lesssim (\mu(K))^{(N+2)/N}\left(1+\log_+\left((\mu(K))^{-1}\right)\right)
\\& \lesssim (\mu(K))^{(N+2)/N} \log\left(\frac{|\tilde{Q}_{200}(0,0)|}{|K|}\right).
\end{align*}
Set $\tilde{\mu}=\left(\log\left(\frac{|\tilde{Q}_{200}(0,0)|}{|K|}\right)\right)^{-N/(N+2)}\mu/\mu(K)$, then
$
||\mathbb{I}_2^1[\tilde{\mu}]||_{L^{(N+2)/N}(\mathbb{R}^{N+1})}\lesssim 1.$\\
It is well known that 
\begin{align}\label{6h0206201410}
\text{Cap}_{2,1,\frac{N+2}{2}}(K)\asymp \sup\{(\omega(K))^{(N+2)/2}:\omega\in\mathfrak{M}^+(K),||\mathbb{I}_2^1[\omega]||_{L^{(N+2)/N}(\mathbb{R}^{N+1})}\lesssim 1\}
\end{align}
see \cite[Section 4]{66H1}.
This gives the second inequality in \eqref{6h010620141}. \\
It is easy to prove \eqref{6h010620149} from its definition.  Moreover,  \eqref{6h0206201410} implies that 
\begin{align*}
\frac{1}{\text{Cap}_{2,1,\frac{N+2}{2}}(K)^{2/N}}\asymp \inf\{||\mathbb{I}_2^1[\omega]||_{L^{(N+2)/N}(\mathbb{R}^{N+1})}^{(N+2)/N}:\omega\in\mathfrak{M}^+(K), \omega(K)=1\}.
\end{align*} 
We deduce from \eqref{6h0206201411} that 
\begin{align}\label{6h0206201412}
\frac{1}{\text{Cap}_{2,1,\frac{N+2}{2}}(K)^{2/N}}\asymp \inf\left\{\int_{\mathbb{R}^{N+1}}\int_{0}^{1}(\omega(\tilde{Q}_r(x,t)))^{2/N}\frac{dr}{r}d\mu(x,t):\omega\in\mathfrak{M}^+(K), \omega(K)=1\right\}.
\end{align}
As in \cite[proof of Lemma 2.2]{66Lab}, it is easy to derive \eqref{6h0106201410} from \eqref{6h0206201412}. 
\end{proof} \medskip
%%%%%%%%%%%%%%%%%%%%%%%%%%%%%%%%%%%%%%%%%%%%%%%%%%%%%%%%%%%%%%%%%%%%%%%%%%%%%%%%%%%%%%%%%%%%%%%%%%%%%%%%%%%%%%%%%

\noindent\begin{proof}[Proof of Proposition \ref{6h300520142}] Thanks to the Poincar\'e inequality,
it is enough to show that there exists $\varphi\in C_c^\infty(\tilde{Q}_{3/2}(0,0))$ such that $0\leq \varphi \leq 1$, with $\varphi=1$ in an open neighborhood of  $K$ and 
\begin{align}
\int_{\mathbb{R}^{N+1}}(|D^2 \varphi|^p+|\partial_t\varphi|^p)dxdt\lesssim \text{Cap}_{2,1,p}(K). 
\end{align}
By definition, one can find $0\leq \phi \in S(\mathbb{R}^{N+1})$,  $\phi\geq 1$ in a neighborhood of $K$ such that 
\begin{align*}
\int_{\mathbb{R}^{N+1}}(|D^2 \phi|^p+|\nabla \phi|^p+|\phi|^p+|\partial_t\phi|^p)dxdt\leq 2\text{Cap}_{2,1,p}(K).
\end{align*}
Let $\eta$ be a cut off function on $\tilde{Q}_{1}(0,0)$ with respect to $\tilde{Q}_{3/2}(0,0)$ and $H\in C^\infty (\mathbb{R})$ such that 
\begin{align*}
0\leq H(t)\leq t^+,~ |t||H^{\prime\prime}(t)|\lesssim 1 ~\text{ for all } t\in \mathbb{R}, ~H(t)=0~\text{ for  } t\leq 1/4~~\text{ and } ~H(t)=1 ~\text{ for  } t\geq 3/4.
\end{align*}
We claim that 
\begin{align}\label{6h300520143}
\int_{\mathbb{R}^{N+1}}(|D^2 \varphi|^p+|\partial_t\varphi|^p)dxdt\lesssim \int_{\mathbb{R}^{N+1}}(|D^2 \phi|^p+|\nabla \phi|^p+|\phi|^p+|\partial_t\phi|^p)dxdt,
\end{align}
where $\varphi =\eta H(\phi)$. Indeed, we have 
$$
|D^2 \varphi|\lesssim |D^2\eta|H(\phi)+|\nabla\eta| |H^{\prime}(\phi)||\nabla \phi|+\eta |H^{\prime\prime}(\phi)||\nabla \phi|^2+\eta|H^{\prime}(\phi)||D^2\phi|,$$
and
$$
|\partial_t\varphi|\lesssim |\partial_t\eta| H(\phi)+\eta |H^{\prime}(\phi)||\phi_t|, ~~H(\phi)\leq \phi, ~~\phi|H^{\prime\prime}(\phi)|\lesssim 1.
$$
Thus, 
\begin{align*}
\int_{\mathbb{R}^{N+1}}(|D^2 \varphi|^p+|\partial_t\varphi|^p)dxdt&\lesssim \int_{\mathbb{R}^{N+1}}(|D^2 \phi|^p+|\nabla \phi|^p+|\phi|^p+|\partial_t\phi|^p)dxdt\\& 
~~~+ \int_{\mathbb{R}^{N+1}}\frac{|\nabla \phi|^{2p}}{\phi^p}dxdt.
\end{align*}
This implies \eqref{6h300520143} since, according to \cite{66Ad1}, one has
 \begin{align*}
 \int_{\mathbb{R}^{N}}\frac{|\nabla \phi(t)|^{2p}}{(\phi(t))^p}dx\lesssim \int_{\mathbb{R}^N}|D^2 \phi(t)|^pdx~~\forall t\in\mathbb{R}.
 \end{align*}
 
\end{proof}
 
\end{document}